\documentclass[12pt]{article}
\usepackage[margin=1in]{geometry}
\usepackage{graphicx}
\usepackage{multirow}%
\usepackage{amsmath,amssymb,amsfonts,amsthm}
\usepackage{hyperref}
\usepackage{cleveref}
\usepackage{mathrsfs}%
\usepackage[title]{appendix}%
\usepackage{xcolor}%
\usepackage{textcomp}%
\usepackage{manyfoot}%
\usepackage{booktabs}%
\usepackage{algorithm}%
\usepackage{algorithmicx}%
\usepackage{algpseudocode}%
\usepackage{listings}%
\usepackage{fixdif}

\usepackage{longtable}
\usepackage{multirow}
\usepackage{lscape}
\usepackage{subcaption}
\usepackage[width=.75\textwidth]{caption}
\usepackage{bm}
\usepackage{bbold}
\usepackage{cases}
\usepackage{authblk}
\usepackage{tikz}
\allowdisplaybreaks

\newtheorem{theorem}{Theorem}
\newtheorem{proposition}[theorem]{Proposition}%
\newtheorem{lemma}[theorem]{Lemma}%
\newtheorem{corollary}[theorem]{Corollary}%

\theoremstyle{thmstylethree}%

\newcommand{\mtc}[1] {\mathcal{#1}}

\begin{document}

\title{Production planning under demand and endogenous supply uncertainty}

\author[1]{Mike Hewitt}
\author[2]{Giovanni Pantuso}

\affil[1]{Department of Information Systems and Supply Chain Management, Loyola University Chicago, USA \\
  \href{mailto:mhewitt3@luc.edu}{mhewitt3@luc.edu}  
}
\affil[2]{Department of Mathematical Sciences, University of Copenhagen, Denmark \\
  \href{mailto:gp@math.ku.dk}{gp@math.ku.dk}}

\date{}

\maketitle

\begin{abstract}
We study the problem of determining how much finished goods inventory to source from different capacitated facilities in order to maximize profits resulting from sales of such inventory. We consider a problem wherein there is uncertainty in demand for finished goods inventory and production yields at facilities. Further, we consider that uncertainty in production yields is endogenous, as it depends on both the facilities where a product is produced and the volumes produced at those facilities. We model the problem as a two stage stochastic program and propose an exact, Benders-based algorithm for solving instances of the problem. We prove the correctness of the algorithm and with an extensive computational study demonstrate that it outperforms known benchmarks. Finally, we establish the value in modeling uncertainty in both demands and production yields.
\end{abstract}

\section{Introduction}\label{sec1}
In many industries, organizations sell finished goods inventories of multiple products to customers to make a profit. To support such sales, these organizations face the production planning problem of determining how much finished goods inventory of these products to source from multiple facilities. Such a planning problem is often complicated by uncertainty in supply, demand, or both. On the supply side, the organization may place a production order with a facility but only a percentage of items produced are of sufficient quality to be sold. In addition, that percentage, often referred to as a \textit{yield}, is not known with certainty. On the demand side, the organization may not know with certainty how much demand there will be for finished goods of each product. 

When statistical distributions for uncertain parameter values (demands, yields) are available, and production decisions do not impact those distributions (i.e. they are exogenous to planning decisions), such a planning problem is a natural candidate for a two-stage stochastic programming model. In such a model, production decisions at facilities are made in the first stage before yield and demand values are revealed. After they are revealed, the second stage models meet known demands for products with known amounts of finished goods inventory. The goal of the stochastic program is to maximize expected profits. While easy to state, algorithmic developments may be needed for solving large-scale instances of such a model. 

However, there are many industrial situations in which demand uncertainty is exogenous but production planning decisions do impact supply uncertainty. For example, consider a company that relies on third party contract manufacturers for some or all of its finished goods inventories. Some contract manufacturers may have better production facilities and processes than others. In such a case, while the type of distribution (e.g. Normal) that best describes production yields may be the same for all manufacturers, the distribution means for some manufacturers may be higher. Or, suppose the company has a longer history of ordering finished goods from some manufacturers than others. In such a case, the type of distribution and its mean that best describes production yields  may be the same across all manufacturers. However, the variances are smaller for the ones with whom the company has more experience. In both these examples, the choice of contract manufacturer(s) that produce a product impacts the supply uncertainty faced.  

Alternately, suppose a manufacturer gets better at producing products of sufficient quality the more it produces. In such a case the mean of a yield distribution at a facility may increase the larger the finished goods order placed with that facility.  Or, suppose orders have deadlines and the yield represents the percentage of the order that is completed by the deadline. The literature on learning curves \cite{ANZANELLO2011573} has shown that the more often an individual performs a task the less time they require to complete it. Thus, the mean of this yield distribution will depend on the size of the production order and in multiple ways. Conversely, the literature \cite{Moutaz1994,SANA2010158} has also identified cases in which yields go down as production volumes (or rates) increase. This could be due to factors such as machine wear or worker fatigue. In all these cases, the yield distribution at a facility may depend on the finished goods order placed with that facility.

A practical example of the problem we consider occurs in agricultural planning \cite{jones2001matching,jones2003managing,ahumada2009application,syngenta_cover}. Namely, the determination of how much to plant of different crops in different fields. Due to the long time that elapses between when a field is planted and the crop harvested, demand is typically only known in distribution. Similarly, due to multiple effects, including weather, the yield of a crop in a field is often uncertain as well. There has been increased adoption \cite{9167626,8643194,9718402} of technology to measure factors that are believed to impact crop yields, and in real time. This data can then be used to more accurately forecast crop yield and its dependence on those factors \cite{JIN2018141}.  Further, the production costs that reflect acquiring seeds to plant and the planting process itself cannot be recovered. Because of different climates, the distributions of harvest yields may vary by field. Thus, the fields in which a crop is planted impact the uncertainty in how much is harvested.  As a grower's ability to effectively harvest a crop tends to improve the more of it they plant, the distributions of the yield in a given field may depend on the amount planted.  Thus, how much crop is planted in a chosen field impacts the uncertainty in how much is harvested.

Motivated by these examples (and others) we consider a production planning problem in which demand uncertainty is exogenous but supply uncertainty is endgenous.
We  assume that both demand distributions and yield distributions can be approximated to a sufficient level of accuracy by a finite set of scenarios.
We model the problem as a two-stage stochastic program in which the first stage of the model contains the same decisions as a stochastic program that models exogenous uncertainty in demand and supply.
However, we explicitly model that the resulting yield distributions and set of appropriate scenarios to consider in the second stage depends on those first stage decisions.
As doing so leads to a very large number of scenarios we propose an exact Benders-based solution method for solving instances of the problem.
With an extensive computational study we demonstrate the value in modeling supply and demand uncertainty separately as well as together.
We also demonstrate computationally that the proposed Benders-based scheme outperforms two known benchmarks. 

We believe this paper makes the following contributions.
\begin{itemize}
    \item We introduce a mathematical model for a production planning problem which recognizes exogenous uncertainty in demand and endogenous uncertainty in yield. The model we propose is general and can be adapted to different manufacturing settings and practical considerations in those settings. 
    \item We propose an exact Benders-based decomposition method and prove its correctness. With an extensive computational study we demonstrate that it significantly outperforms known benchmarks. 
    \item We substantiate the need for modeling both yield and demand uncertainty, offering estimates of the increase in expected profits when doing so. Further, we show that modeling both sources of uncertainty leads to a much greater increase in expected profits than modeling only one.
\end{itemize}
The rest of this paper is organized as follows. Section \ref{sec:lit} reviews relevant literature while Section \ref{sec:problem_model} presents a detailed description of the problem we consider and a model of that problem. Section \ref{sec:benders} presents the Benders-based solution method we propose along with additional speed-up techniques. Section \ref{sec:comp_analysis} describes how our computational study was performed and discusses results regarding the value of modeling uncertainty and the performance of the proposed Benders-based method as compared to two benchmarks. Finally, Section \ref{sec:conclusion_future} concludes the paper and presents avenues for future research. 

\section{Literature Review}\label{sec:lit}

In this section, we first discuss relevant literature on optimization-based methods for production planning. We then discuss the literature on stochastic programming.

Starting with the earliest contribution by \cite{ModH55} the literature on mathematical programming models for production planning is extensive \cite{KemKU11}. Given that, we focus our discussion on papers that are particularly relevant to the problem considered in this paper. The interested reader can refer to available textbooks, e.g.,  \cite{JohM74,HaxC84,VosW06,KemKU11}, and surveys, e.g., \cite{GelV81,YanL95,GroG04,MulPGL06,DiaMP14,JamYFXJ19}.

Despite its ability to provide mathematically richer formulations \cite{MisU11} stochastic programming approaches have not reached similar extensive application as in areas such as energy planning and financial services \cite{KemKU11}. As a result, a large portion of the existing research studies deterministic problems \cite{JamYFXJ19}. The majority of the articles dealing with uncertainty consider uncertainty in demand \cite{DiaMP14}. Other sources of uncertainty include, e.g., costs \cite{LeuN07}, input quality \cite{DenFS10,GuaP11}, and yield \cite{YanL95}.

Among these, we find a number of germane studies which address yield and demand uncertainty, possibly simultaneously. \cite{Kaz04} studies production planning with random yield and demand. Particularly, sale price and purchasing cost are increasing with decreasing random yield. The paper focuses on the olive oil industry. 
Similarly, \cite{KazW11} study the role of the yield-dependent trading cost structure influencing the optimal choice of the selling price and production quantity in the agricultural industry. 
\cite{XiaYX15} considers remanufacturing and pricing decisions when both the remanufacturing yield and the demand for remanufactured products are random.
\cite{KanZ18} study a single-item multi-period disassembly scheduling problem with random yields and demands in which procured, returned items are disassembled into components to satisfy their demands. 
\cite{XieMG21} consider demand and yield uncertainty in a study of buyback contracts for a two-echelon supply chain consisting of a buyer and a seller. 
\cite{XiaGYL21} consider a supply chain consisting of a manufacturer, a reliable supplier, and an unreliable supplier with random yield. The authors model the interaction between the actors and derive optimal production quantity of the suppliers and ordering quantity of the manufacturer. 
\cite{KouXY21} study production and pricing decisions of a risk-averse monopoly firm  under supply random yield. They investigate the impact of the firm’s risk-aversion level on its optimal decisions and the corresponding profit. 
\cite{DonGXY22} study the sourcing of a monopoly firm that procures from multiple unreliable suppliers to meet deterministic and price-dependent demand. The production processes of the suppliers are however unreliable and modeled by correlated proportional random yields. 

Our contribution distinctly departs from available methods as it explicitly incorporates both demand and  yield uncertainty, where the specification of the yield uncertainty is endogenously determined by production decisions. This new trait complicates the problem substantially as it prevents the use of classical stochastic programming techniques. Particularly, the resulting problem is adequately rendered in mathematical terms as a stochastic program with endogenous uncertainty, a type of problem where, as we explain next, methodological advancements are still sparse. Some of these methodological advancements are part of our contribution.

From a methodological point of view, following \cite{GoeG06}, there exist at least two ways in which decisions can influence the the nature of underlying stochastic process in stochastic programs.
The first possibility is that decisions alter the probability space underlying the stochastic process, 
thus changing the likelihood of the possible events or the nature of the event space. 
The second possibility is that decisions determine the time when the uncertainty is resolved. 
The problem dealt with in this article is concerned with the first type of endogenous uncertainty.

The research dealing with decisions influencing probability distributions is rather sparse. 
One of the first approaches dates back to \cite{Pfl90} who consider a Markovian random process whose transition probabilities depend on the decision variables of the optimization problem. The author provides an algorithm that converges with probability one. 
Later, \cite{JonWW98} consider the case where decisions influence both the probability measure and the timing of the observation. 
Their framework includes both two- and multi-stage problems. Nevertheless, the decisions that have an impact on the uncertainty are entirely made at the first decision stage. As in the case described in this article, the authors assume that the set of probability measures which can be enforced by decisions is finite and countable. The authors show that the problem can be recast as choosing the best among the stochastic programs determined by a choice of a distribution and 
propose an implicit enumeration algorithm.
In a similar framework, \cite{Pan21} consider multi-stage problems and extends the model of \cite{JonWW98} by allowing decisions at all stages to determine the probability measure for the later stages. The author offers a solution method for a special case and a general purpose multi-distribution scenario tree structure and a mathematical formulation that avoids explicit statement of non-anticipativity constraints which are typically linked to model size growth, see e.g., \cite{ApaG16,HooM16,MooM18}. 
\cite{Hel16} and \cite{HelBT18} discuss several ways of modeling the interplay between the decision variables and the parameters of the underlying probability distributions in two-stage stochastic programs. Particularly, the authors formulate two-stage models where prior probabilities are distorted
through affine transformations, or combined using convex combinations of several probability distributions.
Furthermore, the authors present models which incorporate the parameters of the probability distribution as
first-stage decision variables.

Beyond this, the literature presents a number of practical applications.
\cite{Ahm00} illustrates examples of problems with endogenous uncertainty, such as facility location, network design and
server selection. The author also presents an exact solution method for the resulting one-stage integer problems.
\cite{VisSF04} consider the problem of investing in strengthening actions for the links of a network subject to
disruptive events. The problem is modeled as a two-stage stochastic program where first-stage investment decisions
influence the likelihood of disruptive events happening at upgraded links. The problem is solved using an approximate solution procedure. 
The same problem is studied also by
\cite{Fla10}, \cite{PeeSGV10} and \cite{LauPK14}. 
\cite{HelW05} consider the problem of interdicting a stochastic network, that is a network whose structure is unknown to the interdictor. 
In this problem, the probabilities of different future network configurations depend on previous interdiction actions.
\cite{TonFR12} present an oil refinery planning problem considering that the uncertainty in product
yield is influenced by operation mode changeovers.  Finally, \cite{EscGMU18} study the problem of mitigating the effects of
natural disasters through preventive actions. The problem is formulated as a three-stage stochastic
bilinear integer program with both exogenous and endogenous uncertainty. Particularly, decisions can influence
both the probabilities and the intensity of future uncertain events. 

When the second type of uncertainty is considered, we are in the presence of random processes whose resolution time depends on the decisions made. 
A typical example is provided by \cite{GoeG04}: The decision maker has to decide which gas reservoirs to explore, and when, by installing exploration facilities. The size and quality of the reservoirs is uncertain and the uncertainty is resolved only after facilities have been installed. The literature dealing with this type of stochastic programs includes \cite{ColM08}, \cite{TarG08}, \cite{TarGG09}, \cite{ColM10}, \cite{GupG11}, \cite{MerV11}, \cite{TarGG13}, \cite{ApaG16}.

The problem studied in this article is in line with the work of \cite{JonWW98} and \cite{Pan21} in that the set of potential probability distributions enforced by first-stage decisions is finite and countable.
Particularly, \cite{JonWW98} propose an implicit enumeration algorithm which relies on computing and storing bounds for each choice of a probability distribution. As the authors acknowledge, this solution strategy is viable only when the set of probability measures has a small cardinality.
As it will be more evident in \Cref{sec:problem_model}, this is not the case in the problem at hand. In our case, the number of probability measures grows exponentially with the size of the problem, particularly with the number of production levels. In addition, capacity constraints on the facilities impose a mutual dependence between the probability distributions enforced on different products so that enumaration becomes impractical. To overcome these methodological limitations, we propose an exact Benders-based decomposition strategy based on an ad-hoc optimality cuts and strong valid inequalities. \Cref{sec:comp_analysis} proves the effectivenes of the proposed algorithm.

\section{Problem description and mathematical model}
\label{sec:problem_model}
In this section, we present a formal description of the problem we consider. We then present a mathematical model of that problem. 

\subsection{Problem description}
\label{subsec:problem}
\noindent
We consider a company that sources and sells multiple products. These products can be sourced from one or more capacitated production facilities. These facilities may be manufacturing plants owned and operated by the company or  external suppliers. Each facility has a capacity with respect to total volume it can produce across all products.  It may be that a product cannot be sourced from every facility. When a facility can supply a product, there is a per-unit cost associated with doing so. This cost can vary by product and by facility. 

Finished units of a product for which there is demand can be sold at a pre-determined price that varies by product. Finished units of a product for which there is not demand can be sold at a discounted price that also varies by product. We note that this discounted price is less than the cost of manufacture at any facility. The company makes decisions regarding how many units to source of each product and from which facilities to source those units with uncertainty regarding product demand and supply. It makes these decisions with the goal of maximizing expected profit.

Specifically, the manufacturer faces two sources of uncertainty. The first is the demand for each product. Product demands can be assumed to follow known statistical distributions with known parameter values for those distributions. In addition, demands for different products may be correlated (either positive or negative). However, demand for a product does not depend on the number of units available for sale of that product. In other words, uncertainty in product demands is exogenous. 

The second source of uncertainty relates to the yield of each product at each facility. Namely, while the manufacturer may supply a facility with the raw materials to make a given number of units of a product, the manufacturing process may in fact yield a smaller number of finished goods. These yields are not known with certainty. Instead, probability distributions, and their parameter values, are known. In addition, there may be correlation between the yields of a given product at different facilities.  Finally, yield distributions depend in part on which of the finite levels of production is chosen for a product at each facility, wherein associated with a production level are lower and upper limits on production volume. Note that the yield distributions for a product may be different for different facilities. 

\subsection{Mathematical model}
\label{subsec:model}
We next present a mathematical model of the problem we consider.
Consider a manufacturer that produces a set of products $\mathcal{P}$ at a set of facilities $\mathcal{F}$.
The manufacturer seeks to determine how much of each product $p \in \mathcal{P}$ to make at each facility $f \in \mathcal{F}$.
For each facility $f$ and product $p$ the manufacturer chooses a production level $l$ that defines a lower limit,
$L_{pfl},$ and upper limit, $U_{pfl},$ on the production volume for that product at that facility.
We let $\mathcal{L}_{pf}$ denote the set of potential production levels for product $p$ at facility $f$.
We assume that, for each product $p$ and facility $f$, the production levels determine disjoint production intervals.
That is $[L_{pfl},U_{pfl}]\cap [L_{pfl'},U_{pfl'}]=\emptyset$ for all $l,'l\in\mathcal{L}_{pf}$. 
We presume that the sets $\mathcal{L}_{pf}$ include a level that models  the company not producing product $p \in \mathcal{P}$ at facility $f \in \mathcal{F}$. Namely, a level $l$ such that $L_{pfl}=U_{pfl}=0$. For each facility $f \in \mathcal{F}$ there is a maximum, $B_{f}$, on the total volume it can be allocated. 

We recall that when planning production, the yield of each product at each facility, and the cumulative demand of each product are uncertain to the decision maker. 
As such, we let $\bm{\xi}_p=(\mathbf{Y}_{p,f_1},\ldots,\mathbf{Y}_{p,f_{|\mathcal{F}|}},\mathbf{D}_p)$ denote a random variable representing the uncertain yield of product $p$ at the different facilities, $\mathbf{Y}_{pf}$, and its cumulative demand, $\mathbf{D}_p$. Observe that we use bold fonts for random variables (e.g., $\bm{\xi}_p$) and plain fonts for their realizations (e.g., $\xi_p$). We assume that, for each product $p$, only finitely many probability distributions can be determined and let $\mathcal{D}_p$ denote the set of  joint probability distributions of the yield and demand of product $p$. 
We also presume that there is a map $l:\mathcal{P}\times\mathcal{F}\times \cup_{p\in\mathcal{P}}\mathcal{D}_p\to \cup_{f\in\mathcal{F}}\mathcal{L}_f$ that, for each product $p$ and facility $f$, indicates the production level of product $p$ at facility $f$ that leads to distribution $d \in \mathcal{D}_p$. 
We note that we do not presume this mapping is one-to-one. Namely, for a given product $p$ and facility $f$ there may be multiple distributions that map to the same production level. Rather, it is the specific combination of production levels at the different facilities that determines a probability distribution.  However, the mapping does preclude  multiple production levels $l,l'$ of product $p$ at facility $f$ requiring the same distribution $d.$ If required, this case can be accommodated by creating a second distribution $d'$ which is identical to $d$ and having $l(p,f,d)$ imply production level $l$ and $l(p,f,d')$ imply production level $l'.$
Recall that we presume correlations between yields for the same product at different facilities may exist.
As the yield distribution at a specific facility also depends on the production level chosen at that facility,
for product $p$ it is necessary that the production level indicated by the map $l$ is chosen for each facility $f$ for probability distribution $d\in\mathcal{D}_p$ to accurately represent uncertainty in yields.
We assume that for each pair of products $(p_1,p_2)$ the respective yields $\mathbf{Y}_{p_1,f}$ and $\mathbf{Y}_{p_2,f}$, at the same or at different facilities $f$, are mutually independent.
Finally, we also note that in our mathematical formulation the distribution of demand is endogenously determined by the choice of production level. This entails that the model is also able to handle endogeneity of demand, should that present itself. In the cases we address we simply assume that each distribution has the same marginal distribution of the demand.

We further assume that each probability distribution is either discrete or can be adequately represented by a finite number of scenarios. We let $\mathcal{S}_{pd}$ represent the set of scenarios in distribution $d \in \mathcal{D}_p$ for product $p \in \mathcal{P}$. Associated with scenario $s \in \mathcal{S}_{pd}$ for distribution $d \in \mathcal{D}_{p}$ we let $\pi_{sd}$ denote the probability that it occurs. Regarding a specific scenario $s \in \mathcal{S}_{pd}$, we let $Y_{pfds}$ be the realization of the yield of product $p$ at facility $f$ under scenario $s\in\mathcal{S}_{pd}$ of distribution $d\in\mathcal{D}_{p}$. Similarly, we let $D_{pds}$ represent the realization of the demand for product $p$ under scenario $s\in\mathcal{S}_{pd}$ of distribution $d\in\mathcal{D}_p$.

\begin{figure}
    \centering

\tikzset{every picture/.style={line width=0.75pt}} 

\begin{tikzpicture}[x=0.75pt,y=0.75pt,yscale=-1,xscale=1]

\draw   (50,40) -- (310,40) -- (310,96) -- (50,96) -- cycle ;
\draw    (69.41,96) -- (51.47,138) ;
\draw    (69.41,96) -- (87.94,138) ;
\draw   (51.47,138) -- (62.94,160) -- (40,160) -- cycle ;
\draw   (87.94,138) -- (100,160) -- (75.88,160) -- cycle ;
\draw    (138.53,96) -- (120.59,138) ;
\draw    (138.53,96) -- (157.06,138) ;
\draw   (109.59,149) .. controls (109.59,142.92) and (114.51,138) .. (120.59,138) .. controls (126.66,138) and (131.59,142.92) .. (131.59,149) .. controls (131.59,155.08) and (126.66,160) .. (120.59,160) .. controls (114.51,160) and (109.59,155.08) .. (109.59,149) -- cycle ;
\draw   (147.06,149.47) .. controls (147.06,143.14) and (152.19,138) .. (158.53,138) .. controls (164.86,138) and (170,143.14) .. (170,149.47) .. controls (170,155.81) and (164.86,160.94) .. (158.53,160.94) .. controls (152.19,160.94) and (147.06,155.81) .. (147.06,149.47) -- cycle ;
\draw    (208.53,96) -- (190,140) ;
\draw    (208.53,96) -- (230,140) ;
\draw   (180,140) -- (200,140) -- (200,160) -- (180,160) -- cycle ;
\draw   (220,140) -- (240,140) -- (240,160) -- (220,160) -- cycle ;
\draw    (278.53,96) -- (260,140) ;
\draw    (278.53,96) -- (300,140) ;
\draw   (260,140) -- (270,150) -- (260,160) -- (250,150) -- cycle ;
\draw   (300,140) -- (310,150) -- (300,160) -- (290,150) -- cycle ;
\draw   (370,40) -- (630,40) -- (630,96) -- (370,96) -- cycle ;
\draw    (389.41,96) -- (371.47,138) ;
\draw    (389.41,96) -- (407.94,138) ;
\draw   (371.47,138) -- (382.94,160) -- (360,160) -- cycle ;
\draw   (407.94,138) -- (420,160) -- (395.88,160) -- cycle ;
\draw    (458.53,96) -- (440.59,138) ;
\draw    (458.53,96) -- (477.06,138) ;
\draw   (429.59,149) .. controls (429.59,142.92) and (434.51,138) .. (440.59,138) .. controls (446.66,138) and (451.59,142.92) .. (451.59,149) .. controls (451.59,155.08) and (446.66,160) .. (440.59,160) .. controls (434.51,160) and (429.59,155.08) .. (429.59,149) -- cycle ;
\draw   (467.06,149.47) .. controls (467.06,143.14) and (472.19,138) .. (478.53,138) .. controls (484.86,138) and (490,143.14) .. (490,149.47) .. controls (490,155.81) and (484.86,160.94) .. (478.53,160.94) .. controls (472.19,160.94) and (467.06,155.81) .. (467.06,149.47) -- cycle ;
\draw    (528.53,96) -- (510,140) ;
\draw    (528.53,96) -- (550,140) ;
\draw   (500,140) -- (520,140) -- (520,160) -- (500,160) -- cycle ;
\draw   (540,140) -- (560,140) -- (560,160) -- (540,160) -- cycle ;
\draw    (598.53,96) -- (580,140) ;
\draw    (598.53,96) -- (620,140) ;
\draw   (580,140) -- (590,150) -- (580,160) -- (570,150) -- cycle ;
\draw   (620,140) -- (630,150) -- (620,160) -- (610,150) -- cycle ;

\draw (62,80) node [anchor=north west][inner sep=0.75pt]  [xscale=0.8,yscale=0.8]  {$d_{1}$};
\draw (132,80) node [anchor=north west][inner sep=0.75pt]  [xscale=0.8,yscale=0.8]  {$d_{2}$};
\draw (202,80) node [anchor=north west][inner sep=0.75pt]  [xscale=0.8,yscale=0.8]  {$d_{3}$};
\draw (272,80) node [anchor=north west][inner sep=0.75pt]  [xscale=0.8,yscale=0.8]  {$d_{4}$};

\draw (382,80) node [anchor=north west][inner sep=0.75pt]  [xscale=0.8,yscale=0.8]  {$d_{1}$};
\draw (452,80) node [anchor=north west][inner sep=0.75pt]  [xscale=0.8,yscale=0.8]  {$d_{2}$};
\draw (522,80) node [anchor=north west][inner sep=0.75pt]  [xscale=0.8,yscale=0.8]  {$d_{3}$};
\draw (592,80) node [anchor=north west][inner sep=0.75pt]  [xscale=0.8,yscale=0.8]  {$d_{4}$};

\draw (42,163.4) node [anchor=north west][inner sep=0.75pt]  [xscale=0.8,yscale=0.8]  {$s^{d_{1}}_{1}$};
\draw (77.88,163.4) node [anchor=north west][inner sep=0.75pt]  [xscale=0.8,yscale=0.8]  {$s^{d_{1}}_{2}$};
\draw (111.12,163.4) node [anchor=north west][inner sep=0.75pt]  [xscale=0.8,yscale=0.8]  {$s^{d_{2}}_{1}$};
\draw (147,163.4) node [anchor=north west][inner sep=0.75pt]  [xscale=0.8,yscale=0.8]  {$s^{d_{2}}_{2}$};
\draw (181.12,163.4) node [anchor=north west][inner sep=0.75pt]  [xscale=0.8,yscale=0.8]  {$s^{d_{3}}_{1}$};
\draw (217,163.4) node [anchor=north west][inner sep=0.75pt]  [xscale=0.8,yscale=0.8]  {$s^{d_{3}}_{2}$};
\draw (251.12,163.4) node [anchor=north west][inner sep=0.75pt]  [xscale=0.8,yscale=0.8]  {$s^{d_{4}}_{1}$};
\draw (291,163.4) node [anchor=north west][inner sep=0.75pt]  [xscale=0.8,yscale=0.8]  {$s^{d_{4}}_{2}$};
\draw (170,62.4) node [anchor=north west][inner sep=0.75pt]  [xscale=0.8,yscale=0.8]  {$F_{1}$};
\draw (21,42.4) node [anchor=north west][inner sep=0.75pt]  [xscale=0.8,yscale=0.8]  {$L^p_{A}$};
\draw (21,70.4) node [anchor=north west][inner sep=0.75pt]  [xscale=0.8,yscale=0.8]  {$L^p_{B}$};
\draw (362,163.4) node [anchor=north west][inner sep=0.75pt]  [xscale=0.8,yscale=0.8]  {$s^{d_{1}}_{1}$};
\draw (397.88,163.4) node [anchor=north west][inner sep=0.75pt]  [xscale=0.8,yscale=0.8]  {$s^{d_{1}}_{2}$};
\draw (431.12,163.4) node [anchor=north west][inner sep=0.75pt]  [xscale=0.8,yscale=0.8]  {$s^{d_{2}}_{1}$};
\draw (467,163.4) node [anchor=north west][inner sep=0.75pt]  [xscale=0.8,yscale=0.8]  {$s^{d_{2}}_{2}$};
\draw (501.12,163.4) node [anchor=north west][inner sep=0.75pt]  [xscale=0.8,yscale=0.8]  {$s^{d_{3}}_{1}$};
\draw (537,163.4) node [anchor=north west][inner sep=0.75pt]  [xscale=0.8,yscale=0.8]  {$s^{d_{3}}_{2}$};
\draw (571.12,163.4) node [anchor=north west][inner sep=0.75pt]  [xscale=0.8,yscale=0.8]  {$s^{d_{4}}_{1}$};
\draw (611,163.4) node [anchor=north west][inner sep=0.75pt]  [xscale=0.8,yscale=0.8]  {$s^{d_{4}}_{2}$};
\draw (490,62.4) node [anchor=north west][inner sep=0.75pt]  [xscale=0.8,yscale=0.8]  {$F_{2}$};
\draw (341,42.4) node [anchor=north west][inner sep=0.75pt]  [xscale=0.8,yscale=0.8]  {$L^p_{C}$};
\draw (341,70.4) node [anchor=north west][inner sep=0.75pt]  [xscale=0.8,yscale=0.8]  {$L^p_{D}$};

\end{tikzpicture}
\caption{Example with two facilities, $F_1$ and $F_2$, with two production level each, ($L_A$, $L_B$) and ($L_C$, $L_D$), respectively, and four possible distributions $d_1$ through $d_4$, represented by nodes of different shapes. Each distribution $d_i$ has two possible realizations, namely $s_1^{d_i}$ and $s_2^{d_i}$. }\label{fig:f1}
\end{figure}
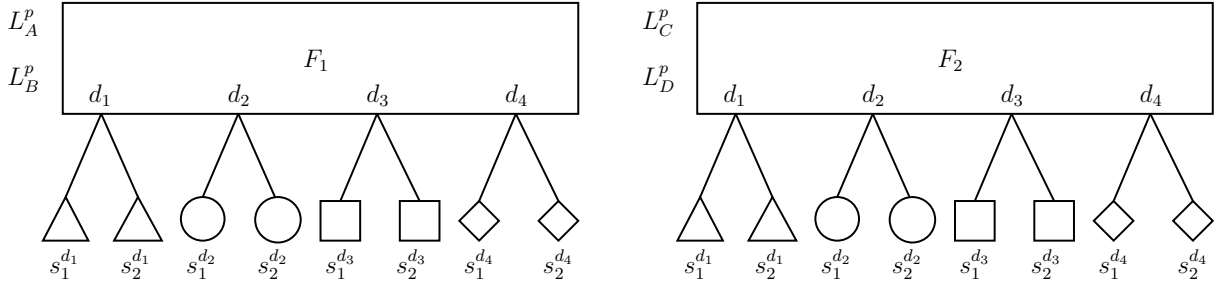

The example illustrated in \Cref{fig:f1} depicts the case of two facilities,  $F_1$ and $F_2,$ that can produce product $p\in\mathcal{P}$. At $F_1$ it is possible to produce product $p$ at levels $L^p_A$ and $L^p_B$. At $F_2$ it is possible to produce at levels $L^p_C$ and $L^p_D$. Thus, in the example, $\mathcal{L}_{F_1,p}=\{L^p_A,L^p_B\}$ and $\mathcal{L}_{F_2,p}=\{L^p_C,L^p_D\}$. 
In addition, there are four alternative probability distributions which may materialize, that is $\mathcal{D}_p=\{d_1,\ldots,d_4\}$.  In the example, the scenarios of different distributions are depicted by different shapes. Note that the shape for a given distribution is the same at different facilities as distribution $d$ describes the yield uncertainty for all facilities. We have $\mathcal{S}_{p,d_i}=\{s_1^{d_i},s_2^{d_i}\}$ for $i=1,\ldots,4$. 
Finally, \Cref{tab:example2} serves to illustrates the map $l(p,f,d)$ for the example in \Cref{fig:f1}.
Considering distribution $d_1$, we have $l(p,F_1,d_1)=L^p_A$ and $l(p,F_2,d_1) = L^p_C$, indicating that for $d_1$ to accurately describe the uncertainty in production yields at the two facilities, production should fall within the bounds associated with level $L_{A}^{p}$ at $F_1$ and the bounds associated with level $L_{C}^{p}$ at $F_2.$
\color{black}

\begin{table}
    \caption{Relationship between production level choices and distribution enforced in the example in \Cref{fig:f1} for each individual product.}
    \label{tab:example2}
    \centering
    \begin{tabular}{cc|c}
    \toprule
    \multicolumn{2}{c}{Level}\\
         $F_1$&$F_2$& Resulting distribution  \\
         \midrule
         $L^p_A$&$L^p_C$&$d_1$  \\
         $L^p_A$&$L^p_D$&$d_2$  \\
         $L^p_B$&$L^p_C$&$d_3$  \\
         $L^p_B$&$L^p_D$&$d_4$  \\
         \bottomrule
    \end{tabular}
\end{table}

Having established how we represent uncertainty, we next introduce a formal mathematical model for the problem described above. Our mathematical model takes the form of a two-stage stochastic program in which the probability distributions describing random elements are determined by the decisions made in the first decision state. 
We let  $x_{pf}\geq 0$ represent the amount of raw materials, quoted in terms of finished goods, of product $p \in \mathcal{P}$ allocated to facility $f \in \mathcal{F}$. The coefficient $C_{pf}$ represents the cost per unit of product $p$ allocated to facility $f$. We let binary variable $y_{pfl}$ equal $1$ if production level $l \in \mathcal{L}_{pf}$ is chosen for product $p \in \mathcal{P}$ at facility $f \in \mathcal{F}$  and $0$ otherwise.  The binary decision variable $\delta_{pd}$ represents whether distribution $d \in \mathcal{D}_{p}$ should be used to describe yield uncertainties for product $p$. Decision variables $x:=(x_{pf})_{p\in\mathcal{P},f\in\mathcal{F}}$, $y:=(y_{pfl})_{p\in\mathcal{P},f\in\mathcal{F},l \in \mathcal{L}_{pf}}$ and $\delta:=(\delta_{pd})_{p\in\mathcal{P},d\in\mathcal{D}_p}$ are first-stage decision variables as the decisions they model are made in advance of yield and demand information being revealed.

Once yield and demand information have materialized through a scenario $s\in\mathcal{S}_{pd}$ for the distribution $d\in\mathcal{D}_p$ for product $p\in\mathcal{P}$, the decision maker makes second-stage decisions concerning sales. We define the continuous variable $z_{pds}$ to represent the amount of available inventory of product $p \in \mathcal{P}$ in scenario $s \in \mathcal{S}_{pd}$ of distribution $d \in \mathcal{D}_{p}.$ The continuous variable $w_{pds}$ defines the amount of product $p \in \mathcal{P}$ that is sold at ``full'' price $P_{p}$ in scenario $s \in \mathcal{S}_{pd}$ of distribution $d \in \mathcal{D}_{p}$ because there is sufficient demand for the product in that scenario at that price. Finally, the continuous variable $o_{pds}$ represents the amount of inventory of product $p \in \mathcal{P}$ that is not sold at price $P_{p}$ in scenario $s \in \mathcal{S}_{pd}$ of distribution $d \in \mathcal{D}_{p}$ due to insufficient demand. Each unit of this inventory is instead sold at a discounted price, $O_p,$ wherein $O_p < C_{pf} \;\; \forall f \in \mathcal{F}.$

With this notation and these decision variables, the decision maker seeks to solve the following optimization problem, which we refer to as the \textit{Production Planning under Demand and Endogenus Supply Uncertainty Problem} (PP-DESUP).
\begin{align}
\label{eq:model:obj}v_{PP-DESUP}^{*} = \max~&-\sum_{p\in\mathcal{P}}\sum_{f\in\mathcal{F}}C_{pf}x_{pf}+\sum_{p\in\mathcal{P}}\sum_{d\in\mathcal{D}_p}\delta_{pd}\Bigg[\sum_{s\in\mathcal{S}_{pd}}\pi_{sd}\bigg(P_p w_{pds}+O_{p}o_{pds}\bigg)\Bigg]
\end{align}
subject to
\begin{align}
\label{eq:model:c0} & \sum_{p \in P} x_{pf} \leq B_{f} & \forall f \in \mathcal{F}, \\
\label{eq:model:c1} &   \sum_{l\in\mathcal{L}_{pf}}y_{pfl} =1&\forall p\in\mathcal{P}, f\in\mathcal{F},  \\
\label{eq:model:c2} &   \sum_{l\in\mathcal{L}_{pf}}L_{pfl}y_{pfl} \leq x_{pf}\leq \sum_{l\in\mathcal{L}_{pf}}U_{pfl}y_{pfl}&\forall p\in\mathcal{P}, f\in\mathcal{F}, \\ 
\label{eq:model:c3}    & \sum_{d\in\mathcal{D}_p}\delta_{pd} = 1 & \forall p\in\mathcal{P}, \\
\label{eq:model:c4} &\sum_{f\in\mathcal{F}}y_{p,f,l(p,f,d)}\geq |\mathcal{F}|\delta_{pd} &\forall p\in\mathcal{P}, d\in \mathcal{D}_p, \\
\label{eq:model:c5}  &z_{pds}= \sum_{f\in\mathcal{F}}Y_{pfds}x_{pf} & \forall p\in\mathcal{P}, d\in\mathcal{D}_p, s\in\mathcal{S}_{pd}, \\
\label{eq:model:c6}  &w_{pds}\leq D_{pds}&\forall p\in\mathcal{P}, d\in\mathcal{D}_p, s\in\mathcal{S}_{pd},\\
\label{eq:model:c7}    &w_{pds} + o_{pds} = z_{pds} & \forall p\in\mathcal{P}, d\in\mathcal{D}_p, s\in\mathcal{S}_{pd}, \\
\label{eq:model:fs_dvx} & x_{pf} \geq 0  &\forall p\in\mathcal{P}, f\in\mathcal{F}, \\ 
\label{eq:model:fs_dvy} & y_{pfl} \in \{0,1\} &\forall p\in\mathcal{P}, f\in\mathcal{F}, l \in L_{pf}, \\
\label{eq:model:fs_dvd} & \delta_{pd} \in \{0,1\} &\forall p\in\mathcal{P}, d \in \mathcal{D}_{p}, \\
\label{eq:model:ss_dvs} & z_{pds} \geq 0, w_{pds} \geq 0, o_{pds} \geq 0 & \forall p \in \mathcal{P}, d \in \mathcal{D}_{p}, s \in \mathcal{S}_{pd}.
\end{align}
The objective (\ref{eq:model:obj}) seeks to maximize expected profit, which is computed as the difference between the expected revenues earned by selling inventory at full or discounted prices and the cost of manufacturing inventory. Constraints (\ref{eq:model:c0}) ensure the total capacity of each facility is observed. Constraints (\ref{eq:model:c1}) ensure that a production level is chosen for each product at each facility. Constraints (\ref{eq:model:c2}) ensures that the amount allocated to a facility falls within the bounds defined by the chosen production level. Constraints (\ref{eq:model:c3}) ensure that a single distribution is used to represent yield and demand uncertainties for each product. Constraints (\ref{eq:model:c4}) ensure that the appropriate production levels are chosen for each product at each facility for the yield distribution used. Constraints (\ref{eq:model:c5}) compute the amount of available inventory of a product in a given scenario. Constraints (\ref{eq:model:c6}) limit the amount of inventory of a product that is sold at full price in a scenario by the demand for that product in that scenario. Constraints (\ref{eq:model:c7}) ensure that each unit of available inventory is sold either at full or discounted price. Constraints (\ref{eq:model:fs_dvx}) - (\ref{eq:model:fs_dvd}) define the first stage decision variables and their domains. Constraints (\ref{eq:model:ss_dvs}) define the second stage decision variables and their domains.

We note that the right-hand-side of constraints (\ref{eq:model:c2}) can be tightened by replacing $U_{pfl}$ with $\min\{B_{f},U_{pfl}\}.$ 
The objective function of the model just presented contains the bilinear terms $\delta_{pd}w_{pds}$ and $\delta_{pd}o_{pds}.$ Since the $\delta$ variables are binary and the $w$ and $o$ variables continuous and bounded, an exact linearization can be obtained using McCormic inequalities, see \cite{Mcc76,AlkF83}. We illustrate the procedure for doing so in \Cref{app:linearized_pp_desup}.

\section{Benders-based method}
\label{sec:benders}
We propose an iterative, Benders-based solution approach to solve instances of the PP-DESUP presented in the previous section. The approach is based on repeatedly solving an optimization problem that is initially a relaxation of the PP-DESUP. This relaxation is formulated with a single variable for each product that serves as an approximation of the expected revenues from selling finished goods inventory of that product. During the solution process the values of approximation variables in a solution to this relaxation can over-estimate the expected revenues possible given the amounts of products allocated to facilities. When an expected revenue over-estimation for a product occurs, the corresponding allocations for that product are used to define an optimality cut that corrects the approximation value. The cut generated for product revenue approximations are added to the relaxation and the solution process repeats. 

In this section, we first present a reformulation of the PP-DESUP that serves as the basis of the relaxation solved by the Benders-based method. We then present that relaxation. After presenting the relaxation, we present optimality cuts that are tight at the point where they are separated and yield a valid upper bound elsewhere. We then present the overall algorithm and a proof that if will converge in finitely many iterations. We finish the section with additional valid inequalities for strengthening the relaxation solved at each iteration of the proposed solution approach.

\subsection{Reformulation}
\label{subsec:mp}
We first reformulate the PP-DESUP based on the observation that computing expected revenues, given values for the $x$ and $y$ variables, is separable by product. Thus, we let $Q_{p}(x,y)$ define the optimal expected revenues generated from product $p \in \mathcal{P}$ given the facility production levels indicated by the $y$ variables and facility allocations prescribed by the $x$ variables. We label this problem, which we refer to as our master problem, as MP. The MP is as follows.

\begin{align}\label{eq:mpv1}
  v_{MP}^{*} = \max  \left\{- \sum_{f \in \mathcal{F}} C_{pf}x_{pf} + \sum_{p \in \mathcal{P}} Q_{p}(x,y) \vert  \eqref{eq:model:c0},\eqref{eq:model:c1},\eqref{eq:model:c2},\eqref{eq:model:fs_dvx},\eqref{eq:model:fs_dvy}\right\}
\end{align}

To compute $Q_{p}(\bar{x},\bar{y})$ for product $p \in \mathcal{P}$ given the values $\bar{y}$ for the $y$ variables and $\bar{x}$ for the $x$ variables we first note that we can determine the appropriate distribution $d(\bar{y}) \in \mathcal{D}_{p}$ given the $\bar{y}$ values. Then, given the distribution $d(\bar{y})$, for each scenario $s \in \mathcal{S}_{pd(\bar{y})}$ in $d(\bar{y})$ and the values $\bar{x}$ for the $x$ variables we can determine the finished goods inventory level  $\bar{z}_{pd(\bar{y})s} = \sum_{f \in \mathcal{F}} Y_{pfd(\bar{y})s}\bar{x}_{pf}.$ As such, given the inventory level $\bar{z}_{pd(\bar{y})s},$ we can determine the revenue for that product, distribution, and scenario by solving the following optimization problem.
\begin{subequations}
\label{eq:spv1}
\begin{align}
    \label{eq:sub:mrh:obj} & Q_{pd(\bar{y})s}(\bar{z}) = \max~ P_p w_{pd(\bar{y})s}+O_{p}o_{pd(\bar{y})s}\\
\label{eq:subproblem:c5}  &w_{pd(\bar{y})s} + o_{pd(\bar{y})s}= \bar{z}_{pd(\bar{y})s} \\
\label{eq:subproblem:c6}  &w_{pd(\bar{y})s}\leq D_{pd(\bar{y})s} \\
\label{eq:subproblem:c7} & w_{pd(\bar{y})s},o_{pd(\bar{y})s} \geq 0 
\end{align}
\end{subequations}
Clearly, $Q_{p}(\bar x, \bar y)$ is feasible for all values of $\bar x, \bar y$ that are feasible for MP. However, we note that this problem can be solved by inspection. Namely, we have
\begin{equation}
   Q_{pd(\bar{y})s}(\bar{z}) = \begin{cases} 
    \bigg(P_p D_{ps} + O_{p}(\bar{z}_{pd(\bar{y})s} - D_{ps}) \bigg) & \mbox{ if } D_{pd(\bar{y})s} \leq \bar{z}_{pd(\bar{y})s} \\
    \bigg(P_p \bar{z}_{pd(\bar{y})s} \bigg) & \mbox{ if } D_{pd(\bar{y})s} > \bar{z}_{pd(\bar{y})s}
    \end{cases}
\end{equation}
Thus, we have $Q_{p}(\bar x, \bar y) = \sum_{s \in\mathcal{S}_{pd(\bar{y})}} \pi_{sd(\bar{y})}  Q_{pd(\bar{y})s}(\bar{z})$ where $\bar{z}_{pd(\bar{y})s} = \sum_{f \in \mathcal{F}} Y_{pfd(\bar{y})s}\bar{x}_{pf}.$ With $Q_{p}(x,y)$ defined in this manner, it is clear that MP is a valid reformulation of PP-DESUP and thus we have $v_{MP}^{*} = v_{PP-DESUP}^{*}.$ We next propose a relaxation of this master problem that is solved in the Benders-based scheme we propose. 

\subsection{Relaxed master problem}
\label{subsec:rmp}
The Benders-based scheme solves an optimization problem that is a relaxation of the MP presented above. To formulate this optimization problem, we let the decision variable $\mu_{p}$, $p \in \mathcal{P}$, serve as an approximation of the expected revenue $Q_{p}(x,y)$ earned from that product. With these additional decision variables, we have the following relaxed master problem, which we refer to as RMP.

\begin{align}
\label{eq:rmpv1}  v^{*}_{RMP} = \max \left\{- \sum_{f \in \mathcal{F}} C_{pf}x_{pf} +  \sum_{p \in \mathcal{P}} \mu_{p} \vert  \eqref{eq:model:c0},\eqref{eq:model:c1},\eqref{eq:model:c2},\eqref{eq:model:fs_dvx},\eqref{eq:model:fs_dvy}, \mu_{p} \geq 0, p \in \mathcal{P}\right\}
\end{align}

Clearly, RMP is a relaxation of the master problem and thus we have $v_{RMP}^{*} \geq v_{MP}^{*}= v_{PP-DESUP}^{*}.$  In addition, it is easy to see that the RMP has an unbounded optimal objective function value as there are no upper bounds on the approximation variables, $\mu_{p}.$  However, we next derive an upper bound on these values that can be added to RMP to ensure it has a finite optimal objective function value. 

First, we observe that for each product $p \in \mathcal{P}$, distribution $d \in \mathcal{D}_{p}$, and scenario $s \in \mathcal{S}_{pd}$ we have that the quantity 
$$\bar{Z}_{pds}=\sum_{f\in\mathcal{F}}Y_{pfds}\min\bigg\{B_f,U_{p,f,l(p,f,d)}\bigg\}$$ 
is an upper bound on the available finished goods inventory for that product in that scenario for that distribution. Namely, we have that $\bar{Z}_{pds} \geq \bar z_{pds},$ the actual finished goods inventory of product $p$ in scenario $s$ of distribution $d$, always holds. The logic behind this claim is as follows.

Consider a given facility $f$ and distribution $d$. Note that as we have a specific distribution $d$ in mind, we know the production level $l(p,f,d)$ for $p$ that must be observed at facility $f$ for distribution $d$ to be in effect. That level in turn induces an upper bound, $U_{p,f,l(p,f,d)}$ on how much of $p$ can be ordered from facility $f.$ Thus, when distribution $d$ is in effect we have that $x_{pf} \leq \min\bigg\{B_f,U_{p,f,l(p,f,d)}\bigg\}$. Supsequently, in a given scenario $s \in \mathcal{S}_{pd}$ we have that the amount of finished goods inventory produced at facility $f$ can be at most $Y_{pfds}\min\bigg\{B_f,U_{p,f,l(p,f,d)}\bigg\}$. As this logic is valid for all facilities, we have our claim.  With $\bar{Z}_{pds}$ providing an upper bound on the finished goods inventory we have that 
$$P_{p}\min\bigg\{D_{ps},\bar{Z}_{pds}\bigg\}+O_{ps}\max\bigg\{\bar{Z}_{pds}-D_{pds},0\bigg\}$$
is an upper bound on the revenue earned on that product in that scenario. Thus, we have the following bound on the expected revenue for product $p \in \mathcal{P}$ over all scenarios and distributions.  
\begin{equation}\label{eq:bnd}
\mu_{p} \leq  \max_{d \in \mathcal{D}_{p}} \bigg\{\sum_{s \in \mathcal{S}_{sd}} \pi_{sd} \bigg(P_{p}\min\big\{D_{pds},\bar{Z}_{pds}\big\}+O_{ps}\max\big\{\bar{Z}_{pds}-D_{pds},0\big\}\bigg)\bigg\}  = M_{p} \;\; \forall p \in \mathcal{P} \tag{Bnd-p}
\end{equation}
We label this upper bound for product $p$ as $M_{p}$. By adding constraints (\ref{eq:bnd}) to RMP we will have $v^{*}_{RMP} < \infty.$ However, it may still be the case that $v^{*}_{RMP} > v^{*}_{MP}.$ Namely, that for the solution $(\bar{x},\bar{y},\bar{\mu})$ to the RMP, the value $\bar{\mu_{p}}$ of the expected revenue approximation variable for product $p$ over-estimates the actual expected revenue $Q_{p}(\bar{x},\bar{y})$. To address this case, we next describe an optimality cut that is added to RMP to render $(\bar{x},\bar{y},\bar{\mu})$ infeasible, effectively correcting the value of the approximation $\mu$. Finally, we note that any solution to RMP yields a feasible second-stage solution $(w,o)$ and is thus not necessary to address feasibility, e.g., by means of so called feasibility cuts.

\subsection{Optimality cuts}
\label{subsec:opt_cut}
We next present the key ingredient of our decomposition, namely an optimality cut we generate from a solution $\big(\bar{x},\bar{y},\bar{\mu}\big)$ that is used to refine the RMP relaxation. 

Before we introduce the optimality cut we set up some necessary additional notation. We recall from \Cref{sec:problem_model} that $l(p,f,d)$ represents the production level $l \in \mathcal{L}_{pf}$ at facility $f$ which is necessary to enforce distribution $d$ on product $p$. Further, we recall that given a solution ($\bar{x},\bar{y}$) to RMP, for product $p \in \mathcal{P}$ we can derive the distribution $d(\bar{y})$ enforced. We can also derive the exact finished goods inventory level $\bar{z}_{pd(\bar{y})s}$ for each scenario in $\mathcal{S}_{pd(\bar{y})}$. Given these inventory levels, the inequality is based on partitioning the set of scenarios $\mathcal{S}_{pd(\bar{y})}$ into two sets. The first, $\mathcal{S}^{\uparrow}_{pd(\bar{y})}(\bar{x},\bar{y}) \subseteq \mathcal{S}_{pd(\bar{y})}$ consists of scenarios $s$ such that $D_{pd(\bar{y})s} < \bar{z}_{pd(\bar{y})s}$. In other words, in these scenarios, there is finished goods inventory in excess of demand. In the second set of scenarios, $\mathcal{S}^{\downarrow}_{pd(\bar{y})}(\bar{x},\bar{y})$ we have $D_{pd(\bar{y})s} \geq \bar{z}_{pd(\bar{y})s}$ meaning all finished goods inventory can be sold at full price. Effectively, $\mathcal{S}_{pd(\bar{y})} = \mathcal{S}^{\uparrow}_{pd(\bar{y})}(\bar{x},\bar{y}) \cup \mathcal{S}^{\downarrow}_{pd(\bar{y})}(\bar{x},\bar{y})$ and $\mathcal{S}^{\uparrow}_{pd(\bar{y})}(\bar{x},\bar{y}) \cap \mathcal{S}^{\downarrow}_{pd(\bar{y})}(\bar{x},\bar{y})=\emptyset$. With these sets we can express the expected revenue of a decision ($\bar{x},\bar{y}$) for all $p \in \mathcal{P}$ as
\begin{align}
\label{eq:eval}
Q_{p}(\bar{x},\bar{y}) =&  \sum_{s \in \mathcal{S}^{\downarrow}_{pd(\bar{y})}(\bar{x},\bar{y})  }  \pi_{sd(\bar{y})} \bigg(P_p \sum_{f \in \mathcal{F}}Y_{pfd(\bar{y})s}\bar{x}_{pf} \bigg)   \\ \nonumber  
+&\sum_{s \in \mathcal{S}^{\uparrow}_{pd(\bar{y})}(\bar{x},\bar{y})} \pi_{sd(\bar{y})}\bigg(O_p\sum_{f\in\mathcal{F}}Y_{pfd(\bar{y})s}\bar{x}_{pf}+(P_p-O_p)D_{pd(\bar{y})s}\bigg)
\end{align}
Observe that in the first term of the right-hand side all inventory is sold at full price while in the second term all inventory is sold at the discounted price $O_{p}$ and only for the amount for which there is demand is the additional $P_{p}-O_{p}$ revenue earned. 

With these items defined, the following propositions introduce an optimality cut and prove its validity.

\begin{proposition}\label{prop:oc:cut}
Let $\big(\bar{x},\bar{y},\bar{\mu}\big)$ be a solution to RMP for which 
$$\bar{\mu}_p > Q_{p}(\bar{x},\bar{y})$$
for some $p\in\mtc{P}$. Then solution $\big(\bar{x},\bar{y},\bar{\mu}\big)$ violates optimality cut 

\begin{align}
\label{eq:single_cut:v1}  \mu_p \leq & \sum_{s \in \mathcal{S}^{\downarrow}_{pd(\bar{y})}(\bar{x},\bar{y})} \pi_{sd(\bar{y})} \bigg(P_p \sum_{f \in \mathcal{F}}Y_{pfd(\bar{y})s}x_{pf} \bigg) \\ \nonumber & +\sum_{s \in \mathcal{S}^{\uparrow}_{pd(\bar{y})}(\bar{x},\bar{y})} \pi_{sd(\bar{y})}\bigg(O_p\sum_{f\in\mathcal{F}}Y_{pfd(\bar{y})s}x_{pf}+(P_p-O_p)D_{pd(\bar{y})s}\bigg)\\
    \nonumber  &+ M_{p}\bigg(\vert \mathcal{F} \vert - \sum_{f \in \mathcal{F}} y_{p,f,l(p,f,d(\bar{y}))}\bigg) 
\end{align}
\end{proposition}
\begin{proof}
  \itshape
By substituting $\big(\bar{x},\bar{y},\bar{\mu}\big)$ in \eqref{eq:single_cut:v1} we observe that the last term of the right-hand side becomes null as $\sum_{f \in \mathcal{F}} y_{p,f,l(p,f,d(\bar{y}))}=|\mathcal{F}|$ for the distribution $d(\bar{y})$ enforced by $\bar{y}$. That is, $\bar{y}_{p,f,l(p,f,d)}=1$ for all $f\in\mtc{F}$, whenever $d$ is the distribution enforced by $\bar{y}$, i.e., $d(\bar{y})$.
Thus, the right-hand side reduces to $Q_{p}(\bar{x},\bar{y})$, see \eqref{eq:eval}, and we obtain 
$$\bar{\mu}_p \leq Q_{p}(\bar{x},\bar{y})$$
proving that the optimality cut is violated by solution  $\big(\bar{x},\bar{y},\bar{\mu}\big)$. 
\end{proof}

Proposition \ref{prop:oc:cut} introduces optimality cut \eqref{eq:single_cut:v1} and effectively allows us to use a Benders-based approach where, upon finding a solution $\big(\bar{x},\bar{y},\bar{\mu}\big)$ where $\bar{\mu}$ overestimates the true expected revenue, one can cut off the solution by adding cut \eqref{eq:single_cut:v1} to RMP. The following corollary clarifies that optimality cut \eqref{eq:single_cut:v1} not only prevents the solution at which it has been generated, but all solutions in a neighborhood of production volumes wherein the relationship between finished goods inventory and product demands remains the same in each scenario. 

\begin{corollary}\label{cor:oc:otherx}
Let $\big(\bar{x},\bar{y},\bar{\mu}\big)$ be a solution to RMP for which 
$$\bar{\mu}_p > Q_{p}(\bar{x},\bar{y})$$
for some $p'\in\mtc{P}$. Assume an optimality cut \eqref{eq:single_cut:v1} is generated and added to RMP.
Let $\big(\hat{x},\bar{y},\hat{\mu}\big)$ be a different solution to RMP (note that the $y$ component is the same as for the previous solution) for which 
$$\hat{\mu}_{p'} > Q_{p'}(\hat{x},\bar{y})$$
and for which 
$$\mathcal{S}^{\uparrow}_{p',d(\bar{y})}(\hat{x},\bar{y}) =\mathcal{S}^{\uparrow}_{p',d(\bar{y})}(\bar{x},\bar{y}) $$
i.e., the change from $\bar{x}$ to $\hat{x}$ does not change the partition of $\mathcal{S}_{p',d(\bar{y})}$.
Then the optimality cut generated for 
$\big(\bar{x},\bar{y},\bar{\mu}\big)$
is violated by $\big(\hat{x},\bar{y},\hat{\mu}\big)$.
\end{corollary}
\begin{proof}
\itshape
  By substituting for $\big(\hat{x},\bar{y},\hat{\mu}\big)$ in the optimality cut added we obtain 
\begin{align*}
\hat{\mu}_{p'} \leq & \sum_{s \in \mathcal{S}^{\downarrow}_{p'd(\bar{y})}(\bar{x},\bar{y})} \pi_{sd(\bar{y})} \bigg(P_{p'} \sum_{f \in \mathcal{F}}Y_{p'fd(\bar{y})s}\hat{x}_{p'f} \bigg) \\ \nonumber 
+~& \sum_{s \in \mathcal{S}^{\uparrow}_{p'd(\bar{y})}(\bar{x},\bar{y})} \pi_{sd(\bar{y})}\bigg(O_{p'}\sum_{f\in\mathcal{F}}Y_{p'fd(\bar{y})s}\hat{x}_{p'f}+(P_{p'}-O_{p'})D_{p'd(\bar{y})s}\bigg)\\
    = &~ Q_{p'}(\hat{x},\bar{y})
\end{align*}
which is violated since $\hat{\mu}_{p'} > Q_{p'}(\hat{x},\bar{y})$. 
\end{proof}
Corollary \ref{cor:oc:otherx} illustrates that once an optimality cut is added to RMP, in addition to cutting off a solution which bears an overestimation of the revenue, it enforces an exact estimate of the expected revenue for all other solutions $x$, provided that the resulting finished goods inventory level yields the same partition of $\mathcal{S}_{pd(\bar{y})}$ as production level $\bar x$. 
We can also conclude from Corollary \ref{cor:oc:otherx} that there are finitely many inequalities \eqref{eq:single_cut:v1}. We observe that for each product $p$ and distribution $d$, the resulting finished goods inventory level in each scenario can either exceed demand or not. Thus, the total number of inequalities \eqref{eq:single_cut:v1} needed to guarantee $v_{RMP} = v_{MP}$ is $\sum_{p\in\mtc{P}}\sum_{d\in\mtc{D}_p}2^{|\mtc{S}_{pd}|}.$

We next prove that inequality \eqref{eq:single_cut:v1} is not violated by a solution to RMP from which an optimal solution to PP-DESUP can be derived. To do so, we consider a solution $(x^*,y^*,\mu^*)$ to RMP for which it holds that $\mu^*_p = Q_{p}(x^*,y^*)$ for some $p \in \mtc{P}.$  Examining the possible relationship between $(\bar{x},\bar{y})$ and $(x^*,y^*)$ we see there are three cases (for the given $p$) to consider. Namely,
\begin{enumerate}
    \item $y_{pfl}^* = \bar{y}_{pfl} \;\; \forall f \in \mathcal{F}, l \in \mathcal{L}_{pf}$ and $x_{pf}^{*} = \bar{x}_{pf} \;\; \forall f \in \mathcal{F}$,    
    \item $y_{pfl}^* \neq \bar{y}_{pfl}$ for some  $f \in \mathcal{F}$ and $l \in \mathcal{L}_{pf}$,
    \item $y_{pfl}^* = \bar{y}_{pfl} \;\; \forall f \in \mathcal{F}, l \in \mathcal{L}_{pf}$ and $x_{pf}^*\neq \bar{x}_{pf}$ for some $f \in \mathcal{F}$.
\end{enumerate}
We next prove that in each of these three cases inequality \eqref{eq:single_cut:v1} is satisfied by $(x^*,y^*,\mu^*)$.

\begin{lemma}
  \label{lem:valid_cut_case_1}
  Consider inequality \eqref{eq:single_cut:v1}, generated for some $(\bar{x},\bar{y},\bar{\mu})$ and some $p\in\mathcal{P}$.
  Let $(x^*,y^*,\mu^*)$ be a solution to RMP for which it holds that
  $$\mu^*_p = Q_{p}(x^*,y^*)$$
  for the same $p\in\mtc{P}$, and where $y_{pfl}^* = \bar{y}_{pfl} \;\; \forall f \in \mathcal{F}, l \in \mathcal{L}_{pf}$ and $x_{pf}^{*} = \bar{x}_{pf} \;\; \forall f \in \mathcal{F}$.    
  Then, inequality \eqref{eq:single_cut:v1} is satisfied by $(x^*,y^*,\mu^*)$.
\end{lemma}
\begin{proof}
  \itshape
Since $y_{pfl}^* = \bar{y}_{pfl}$  $\forall f \in \mathcal{F}$, and $l \in \mathcal{L}_{pf}$ the right-hand side of the optimality cut reduces to 
$$\sum_{s \in \mathcal{S}^{\downarrow}_{pd(\bar{y})}(\bar{x},\bar{y})} \pi_{sd(\bar{y})} \bigg(P_p \sum_{f \in \mathcal{F}}Y_{pfd(\bar{y})s}x^{*}_{pf} \bigg)  +\sum_{s \in \mathcal{S}^{\uparrow}_{pd(\bar{y})}(\bar{x},\bar{y})} \pi_{sd(\bar{y})}\bigg(O_p\sum_{f\in\mathcal{F}}Y_{pfd(\bar{y})s}x^{*}_{pf}+(P_p-O_p)D_{pd(\bar{y})s}\bigg)$$ 
In addition, we have that $d(\bar{y}) = d(y^{*})$ and, as $x_{pf}^{*} = \bar{x}_{pf} \;\; \forall f \in \mathcal{F}$, the sets $\mathcal{S}^{\downarrow}_{pd(\bar{y})}(\bar{x},\bar{y})$ and $\mathcal{S}^{\downarrow}_{pd(\bar{y})}(x^*,y^*)$ are the same (as are their complements). Thus, the right-hand-side is equivalent to the right-hand-side of (\ref{eq:eval}) and we obtain 
$$\mu^*\leq  Q_{p}(x^*,y^*)$$
which is true by assumption. Thus, $(x^*,y^*,\mu^*)$ satisfies the optimality cut.
\end{proof}

\begin{lemma}
  \label{lem:valid_cut_case_2}
  Consider inequality \eqref{eq:single_cut:v1}, generated for some $(\bar{x},\bar{y},\bar{\mu})$ and some $p\in\mathcal{P}$.
  Let $(x^*,y^*,\mu^*)$ be a solution to RMP for which it holds that
  $$\mu^*_p = Q_{p}(x^*,y^*)$$
  for the same $p\in\mtc{P}$, and where $y_{pfl}^* \neq \bar{y}_{pfl}$ for some  $f \in \mathcal{F}$ and $l \in \mathcal{L}_{pf}$.
  Then, \eqref{eq:single_cut:v1} is satisfied by $(x^*,y^*,\mu^*)$.
\end{lemma}
\begin{proof}
  \itshape
We have $\vert \mathcal{F} \vert - \sum_{f \in \mathcal{F}} y^{*}_{p,f,l(p,f,d(\bar{y}))} \geq 1$ and thus the right-hand side of inequality \eqref{eq:single_cut:v1} is at least as great as $M_{p}$ which was already shown to be a valid upper bound on $Q_{p}(x^{*},y^{*})$. Thus, $(x^*,y^*,\mu^*)$ satisfies the optimality cut.

\end{proof}

\begin{lemma}
  \label{lem:valid_cut_case_3}
  Consider inequality \eqref{eq:single_cut:v1}, generated for some $(\bar{x},\bar{y},\bar{\mu})$ and some $p\in\mathcal{P}$.
  Let $(x^*,y^*,\mu^*)$ be a solution to RMP for which it holds that
  $$\mu^*_p = Q_{p}(x^*,y^*)$$
  for the same $p\in\mtc{P}$, and where $y_{pfl}^* = \bar{y}_{pfl}$ for all  $f \in \mathcal{F}$ and $l \in \mathcal{L}_{pf}$, and $x_{pf}^*\neq \bar{x}_{pf}$ for some $f \in \mathcal{F}$.
  Then, inequality \eqref{eq:single_cut:v1} is satisfied by $(x^*,y^*,\mu^*)$.
\end{lemma}
\begin{proof}
  \itshape
This case is more involved as the sets $\mathcal{S}^{\downarrow}_{pd(\bar{y})}(\bar{x},\bar{y})$ and $\mathcal{S}^{\downarrow}_{pd(\bar{y})}(x^*,y^*)$ (and their complements) may contain different scenarios as $\bar{x}$ and $x^*$ are different and so are, in general, the inventories of finished goods. As such, to show that the optimality cut is satisfied, we need to show that the right-hand side of the cut overestimates the true expected revenue of solution ($x^*,y^*$) for product $p$. That is, we  next show that the following inequality holds
\begin{align}
\label{eq:third_case}
\sum_{s \in \mathcal{S}^{\downarrow}_{pd(\bar{y})}(x^{*},y^{*})} \pi_{sd(\bar{y})} \bigg(P_p \sum_{f \in \mathcal{F}}Y_{pfd(\bar{y})s}x^*_{pf} \bigg) 
     +\sum_{s \in \mathcal{S}^{\uparrow}_{pd(\bar{y})}(x^{*},y^{*})}\pi_{sd(\bar{y})}\bigg(O_p\sum_{f\in\mathcal{F}}Y_{pfd(\bar{y})s}x^*_{pf}+(P_p-O_p)D_{pd(\bar{y})s}\bigg)
    \leq \nonumber \\
    \sum_{s \in \mathcal{S}^{\downarrow}_{pd(\bar{y})}(\bar{x},\bar{y})} \pi_{sd(\bar{y})} \bigg(P_p \sum_{f \in \mathcal{F}}Y_{pfd(\bar{y})s}x^*_{pf} \bigg) 
  +\sum_{s \in \mathcal{S}^{\uparrow}_{pd(\bar{y})}(\bar{x},\bar{y})}\pi_{sd(\bar{y})}\bigg(O_p\sum_{f\in\mathcal{F}}Y_{pfd(\bar{y})s}x^*_{pf}+(P_p-O_p)D_{pd(\bar{y})s}\bigg)
\end{align}
Here the left-hand side of inequality \eqref{eq:third_case} gives $Q_p(x^*,y^*)$: simply note that $\mathcal{S}^{\downarrow}_{pd(\bar{y})}(x^{*},y^{*})=\mathcal{S}^{\downarrow}_{pd(y^*)}(x^{*},y^{*})$ (and so their complements) as $d(\bar{y})=d(y^*)$. 
The right-hand side computes expected revenues using the partition of scenarios induced by solution $\bar{x}$. Particularly, note that in the right-hand side the sums are taken over the sets of scenarios $\mathcal{S}^{\downarrow}_{pd(\bar{y})}(\bar{x},\bar{y}), \mathcal{S}^{\uparrow}_{pd(\bar{y})}(\bar{x},\bar{y})$ induced by $(\bar{x},\bar{y},\bar{\mu})$ but with values from the solution $(x^{*},y^{*})$. 
We will show that \eqref{eq:third_case} holds by proving that in the cases where expected revenues are computed differently for a scenario, the method used on the right-hand side of the inequality (i.e., in the optimality cut) is an over-estimate.

To do so, given the two partitions of the set of scenarios $\mathcal{S}_{pd(\bar{y})}$ induced by the two solutions, each into two sets, we further partition the set of scenarios $\mathcal{S}_{pd(\bar{y})}$ into four sets as follows. 
\begin{enumerate}
    \item $\mathcal{S}_{pd(\bar{y})}^A=\{s\in\mathcal{S}_{pd(\bar{y})}|s\in \mathcal{S}^{\downarrow}_{pd(\bar{y})}(x^{*},y^{*})$ and $s\in \mathcal{S}^{\downarrow}_{pd(\bar{y})}(\bar{x},\bar{y}) \}$
    \item $\mathcal{S}_{pd(\bar{y})}^B=\{s\in\mathcal{S}_{pd(\bar{y})}|s\in \mathcal{S}^{\downarrow}_{pd(\bar{y})}(x^{*},y^{*})$ and $s\in \mathcal{S}^{\uparrow}_{pd(\bar{y})}(\bar{x},\bar{y})\}$
    \item $\mathcal{S}_{pd(\bar{y})}^C=\{s\in\mathcal{S}_{pd(\bar{y})}|s\in \mathcal{S}^{\uparrow}_{pd(\bar{y})}(x^{*},y^{*})$ and $s\in \mathcal{S}^{\uparrow}_{pd(\bar{y})}(\bar{x},\bar{y})\}$
    \item $\mathcal{S}_{pd(\bar{y})}^D=\{s\in\mathcal{S}_{pd(\bar{y})}|s\in \mathcal{S}^{\uparrow}_{pd(\bar{y})}(x^{*},y^{*})$ and $s\in \mathcal{S}^{\downarrow}_{pd(\bar{y})}(\bar{x},\bar{y})\}$
\end{enumerate}
\begin{figure}[htp]
    \centering
    \includegraphics[scale=1]{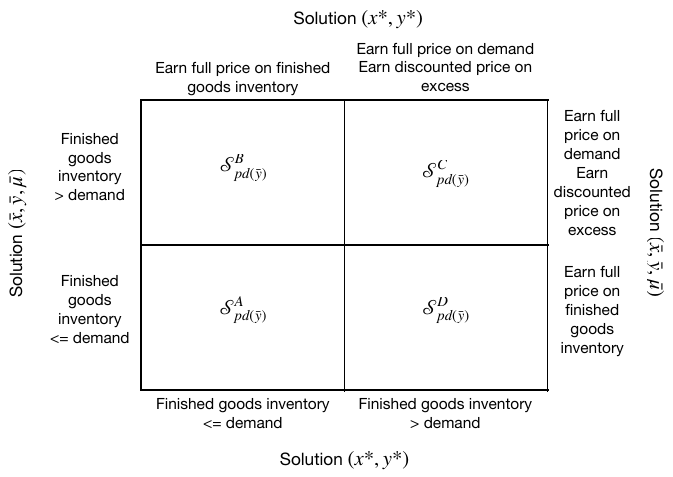}
    \caption{Scenario partition}
    \label{fig:fig_scenario_partition}
\end{figure}
We illustrate this partition in Figure \ref{fig:fig_scenario_partition}, which also depicts how expected revenues would be computed in each set in the partition in each solution. For example, consider a scenario in $\mathcal{S}_{pd(\bar{y})}^A.$ As it is in $\mathcal{S}^{\downarrow}_{pd(\bar{y})}(x^{*},y^{*})$, the expected revenues in the solution $(x^{*},y^{*})$ for this product are computed presuming the full price, $P_p,$ is earned on all finished goods inventory. As the scenario is also in $ \mathcal{S}^{\downarrow}_{pd(\bar{y})}(\bar{x},\bar{y})$, the expected revenues would have been computed in the same manner. Next, consider a scenario in $\mathcal{S}_{pd(\bar{y})}^B.$ Again, the scenario is in $\mathcal{S}^{\downarrow}_{pd(\bar{y})}(x^{*},y^{*})$ and thus the expected revenues in the solution $(x^{*},y^{*})$ for this product are computed presuming the full price, $P_p,$ is earned on all finished goods inventory. However, in the solution $(\bar{x},\bar{y},\bar{\mu})$ to the RMP this scenario was in $\mathcal{S}^{\uparrow}_{pd(\bar{y})}(\bar{x},\bar{y})$. As such, in that solution, expected revenues were computed by first meeting demand with finished goods inventory at full price and then selling whatever inventory is left over at a discount. 

Let us now rewrite \eqref{eq:third_case} using the new partition of the scenarios. We first observe that a scenario $s$ in either $\mathcal{S}_{pd(\bar{y})}^A$ or $\mathcal{S}_{pd(\bar{y})}^C$ evaluates to the same value on each side of inequality \eqref{eq:third_case}. As such, we can rewrite \eqref{eq:third_case} as
\begin{align}
\label{eq:third_case_simplified}
      \sum_{s\in \mathcal{S}^B_{pd(\bar{y})}} \pi_{sd(\bar{y})} \bigg(P_p \sum_{f \in \mathcal{F}}Y_{pfd(\bar{y})s}x^*_{pf} \bigg) 
      +\sum_{s \in \mathcal{S}^D_{pd(\bar{y})}} \pi_{sd(\bar{y})}\bigg(O_p\sum_{f\in\mathcal{F}}Y_{pfd(\bar{y})s}x^*_{pf}+(P_p-O_p)D_{pd(\bar{y})s}\bigg) \leq  \nonumber \\
      \sum_{s \in \mathcal{S}^B_{pd(\bar{y})}} \pi_{sd(\bar{y})}\bigg(O_p\sum_{f\in\mathcal{F}}Y_{pfd(\bar{y})s}x^*_{pf}+(P_p-O_p)D_{pd(\bar{y})s}\bigg)
     +\sum_{s\in \mathcal{S}^D_{pd(\bar{y})}} \pi_{sd(\bar{y})} \bigg(P_p \sum_{f \in \mathcal{F}}Y_{pfd(\bar{y})s}x^*_{pf} \bigg) 
\end{align}

We next compare the terms on each side of inequality \eqref{eq:third_case_simplified} that involve summing over scenarios in $\mathcal{S}^B_{pd(\bar{y})}$.
Observe that the sum over scenarios in $\mathcal{S}^B_{pd(\bar{y})}$ on the right-hand side of \eqref{eq:third_case_simplified} can be rewritten as follows
\begin{align*}
     & \sum_{s \in \mathcal{S}^B_{pd(\bar{y})}}\pi_{sd(\bar{y})}\bigg[ P_pD_{pd(\bar{y})s}+O_p\bigg(\sum_{f\in\mathcal{F}}Y_{pfd(\bar{y})s}x^*_{pf}-D_{pd(\bar{y})s}\bigg)\bigg]\\
     =&\sum_{s \in \mathcal{S}^B_{pd}} \pi_{sd(\bar{y})}\bigg[P_p\bigg(D_{pd(\bar{y})s}+\sum_{f\in\mathcal{F}}Y_{pfd(\bar{y})s}x^*_{pf}-\sum_{f\in\mathcal{F}}Y_{pfd(\bar{y})s}x^*_{pf}\bigg)+O_p\bigg(\sum_{f\in\mathcal{F}}Y_{pfd(\bar{y})s}x^*_{pf}-D_{pd(\bar{y})s}\bigg)\bigg]\\
     =&\sum_{s \in \mathcal{S}^B_{pd(\bar{y})}} \pi_{sd(\bar{y})}\bigg[P_p\bigg(\sum_{f\in\mathcal{F}}Y_{pfd(\bar{y})s}x^*_{pf}\bigg)+P_pD_{pd(\bar{y})s}-P_p\bigg(\sum_{f\in\mathcal{F}}Y_{pfd(\bar{y})s}x^*_{pf}\bigg)+O_p\bigg(\sum_{f\in\mathcal{F}}Y_{pfd(\bar{y})s}x^*_{pf}\bigg)-O_pD_{pd(\bar{y})s}\bigg)\bigg]\\
          =&\sum_{s \in \mathcal{S}^B_{pd(\bar{y})}} \pi_{sd(\bar{y})}\bigg[P_p\bigg(\sum_{f\in\mathcal{F}}Y_{pfd(\bar{y})s}x^*_{pf}\bigg)+P_p\bigg(D_{pd(\bar{y})s}-\sum_{f\in\mathcal{F}}Y_{pfd(\bar{y})s}x^*_{pf}\bigg)-O_p\bigg(D_{pd(\bar{y})s} - \sum_{f\in\mathcal{F}}Y_{pfd(\bar{y})s}x^*_{pf}\bigg)\bigg] \\
          =&\sum_{s \in \mathcal{S}^B_{pd(\bar{y})}} \pi_{sd(\bar{y})}\bigg[P_p\bigg(\sum_{f\in\mathcal{F}}Y_{pfd(\bar{y})s}x^*_{pf}\bigg)+(P_p - O_p)\bigg(D_{pd(\bar{y})s}-\sum_{f\in\mathcal{F}}Y_{pfd(\bar{y})s}x^*_{pf}\bigg)\bigg]
\end{align*}
Next, we recall that scenarios in $\mathcal{S}^B_{pd(\bar{y})}$ are also in $\mathcal{S}^{\downarrow}_{pd(\bar{y})}(x^{*},y^{*})$ and thus we have $D_{pd(\bar{y})s} \geq \sum_{f \in \mathcal{F}}Y_{pfd(\bar{y})s}x_{pf}^{*}$. Given that $P_p>O_p$ we have that
\begin{align*}\
&\sum_{s \in \mathcal{S}^B_{pd(\bar{y})}} \pi_{sd(\bar{y})}P_p\bigg(\sum_{f\in\mathcal{F}}Y_{pfd(\bar{y})s}x^*_{pf}\bigg) \leq \\
&\sum_{s \in \mathcal{S}^B_{pd(\bar{y})}} \pi_{sd(\bar{y})}\bigg[P_p\bigg(\sum_{f\in\mathcal{F}}Y_{pfd(\bar{y})s}x^*_{pf}\bigg)+\underbrace{(P_p - O_p)\bigg(D_{pd(\bar{y})s}-\sum_{f\in\mathcal{F}}Y_{pfd(\bar{y})s}x^*_{pf}\bigg)}_{\geq 0}\bigg]
\end{align*}
showing that the sums over scenarios in $\mathcal{S}^B_{pd(\bar{y})}$ evaluate to a larger value in the right-hand side of \eqref{eq:third_case_simplified} than in the left.  

We next compare the terms on each side of inequality \eqref{eq:third_case_simplified} that involve summing over scenarios in $\mathcal{S}^D_{pd(\bar{y})}$. We observe that the expression on the left-hand side of \eqref{eq:third_case_simplified} that involves summing over scenarios in $\mathcal{S}^D_{pd(\bar{y})}$ can be rewritten as
\begin{align*}
     &\sum_{s \in \mathcal{S}^D_{pd(\bar{y})}} \pi_{sd(\bar{y})}\bigg(O_p\sum_{f\in\mathcal{F}}Y_{pfd(\bar{y})s}x^*_{pf}+(P_p-O_p)D_{pd(\bar{y})s}\bigg)\\
     =&\sum_{s \in \mathcal{S}^D_{pd(\bar{y})}}\pi_{sd(\bar{y})}\bigg[P_p\bigg(D_{pd(\bar{y})s}+\sum_{f\in\mathcal{F}}Y_{pfd(\bar{y})s}x^*_{pf}-\sum_{f\in\mathcal{F}}Y_{pfd(\bar{y})s}x^*_{pf}\bigg)+O_p\bigg(\sum_{f\in\mathcal{F}}Y_{pfd(\bar{y})s}x^*_{pf}-D_{pd(\bar{y})s}\bigg)\bigg]\\
     =&\sum_{s \in \mathcal{S}^D_{pd(\bar{y})}}\pi_{sd(\bar{y})}\bigg[P_p\sum_{f\in\mathcal{F}}Y_{pfd(\bar{y})s}x^*_{pf}+P_pD_{pd(\bar{y})s}-P_p\sum_{f\in\mathcal{F}}Y_{pfd(\bar{y})s}x^*_{pf}+O_p\bigg(\sum_{f\in\mathcal{F}}Y_{pfd(\bar{y})s}x^*_{pf}-D_{pd(\bar{y})s}\bigg)\bigg]\\
     =&\sum_{s \in \mathcal{S}^D_{pd(\bar{y})}}\pi_{sd(\bar{y})}\bigg[P_p\sum_{f\in\mathcal{F}}Y_{pfd(\bar{y})s}x^*_{pf}+(P_p-O_p)D_{pd(\bar{y})s}+(O_p-P_p)\sum_{f\in\mathcal{F}}Y_{pfd(\bar{y})s}x^*_{pf}\bigg]\\
     =&\sum_{s \in \mathcal{S}^D_{pd(\bar{y})}}\pi_{sd(\bar{y})}\bigg[P_p\sum_{f\in\mathcal{F}}Y_{pfd(\bar{y})s}x^*_{pf}+(P_p-O_p)\bigg(D_{pd(\bar{y})s}-\sum_{f\in\mathcal{F}}Y_{pfd(\bar{y})s}x^*_{pf}\bigg)\bigg]\\
\end{align*}
We recall that scenarios in $\mathcal{S}^D_{pd(\bar{y})}$ are also in $\mathcal{S}^{\uparrow}_{pd(\bar{y})}(x^{*},y^{*})$.
As such, we have $\sum_{f \in \mathcal{F}}Y_{pfd(\bar{y})s}x_{pf}^* \geq D_{pd(\bar{y})s},$ or,  $D_{pd(\bar{y})s}- \sum_{f \in \mathcal{F}}Y_{pfd(\bar{y})s}x_{pf}^*\leq 0$. As $P_p>O_p$, we can conclude that 
\begin{align*}
&\sum_{s \in \mathcal{S}^D_{pd(\bar{y})}}\pi_{sd(\bar{y})}\bigg[P_p\sum_{f\in\mathcal{F}}Y_{pfd(\bar{y})s}x^*_{pf}+\underbrace{(P_p-O_p)\bigg(D_{pd(\bar{y})s}-\sum_{f\in\mathcal{F}}Y_{pfd(\bar{y})s}x^*_{pf}\bigg)}_{\leq 0}\bigg]\\
&\leq \sum_{s \in \mathcal{S}^D_{pd(\bar{y})}}\pi_{sd(\bar{y})}\bigg[P_p\sum_{f\in\mathcal{F}}Y_{pfd(\bar{y})s}x^*_{pf}\bigg]
\end{align*}    
where the term on the right-hand side is the corresponding term on the right-hand side of \eqref{eq:third_case_simplified}. This shows that also in the third case the cut is satisfied by the solution and completes the proof. 
\end{proof}

Having addressed each of these cases we are now able to prove that the proposed optimality cut is valid. 
\begin{proposition}
\label{prop:valid_cut}
Let $(x^*,y^*,\mu^*)$ be a solution to RMP for which it holds that
$$\mu^*_p = Q_{p}(x^*,y^*)$$
for some $p\in\mtc{P}$. 
Then, inequality (\ref{eq:single_cut:v1}), generated for some $(\bar{x},\bar{y},\bar{\mu})$ and the same $p$, is satisfied by $(x^*,y^*,\mu^*)$. 
\end{proposition}
\begin{proof}
  \itshape
Consider a solution $(\bar{x},\bar{y},\bar{\mu})$ to RMP and assume it was used to generate cut \eqref{eq:single_cut:v1} for some $p\in\mtc{P}$. Let $(x^*,y^*,\mu^*)$ be a solution to RMP for which
$$\mu^*_p = Q_{p}(x^*,y^*)$$
for the same $p\in\mtc{P}$. There are three possible relationships between $(\bar{x},\bar{y})$ and $(x^*,y^*)$ for the given $p$. They are as follows.
\begin{enumerate}
    \item $y_{pfl}^* = \bar{y}_{pfl} \;\; \forall f \in \mathcal{F}, l \in \mathcal{L}_{pf}$ and $x_{pf}^{*} = \bar{x}_{pf} \;\; \forall f \in \mathcal{F}$,    
    \item $y_{pfl}^* \neq \bar{y}_{pfl}$ for some  $f \in \mathcal{F}$ and $l \in \mathcal{L}_{pf}$,
    \item $y_{pfl}^* = \bar{y}_{pfl} \;\; \forall f \in \mathcal{F}, l \in \mathcal{L}_{pf}$ and $x_{pf}^*\neq \bar{x}_{pf}$ for some $f \in \mathcal{F}$.
\end{enumerate}
Lemmas \ref{lem:valid_cut_case_1}, \ref{lem:valid_cut_case_2}, and \ref{lem:valid_cut_case_3} prove that in each case the solution $(x^*,y^*,\mu^*)$ satisfies the inequality (\ref{eq:single_cut:v1}). As these are the only cases possible, the inequality is valid. 
\end{proof}

Finally, we note the similarity between \eqref{eq:single_cut:v1} and the integer L-shaped method cuts proposed by \cite{LapL93}.
Given a binary first-stage solution $\bar{y}$, the cut generated by \cite{LapL93} reduces to the expected second-stage value of $\bar{y}$ if evaluated in $\bar{y}$, and in a bound otherwise.
Given a first-stage solution $(\bar{x},\bar{y})$, the cut generated by our method reduces to a linear expression in $x$ when evaluated in $\bar{y}$ and in a valid bound otherwise.
The linear expression in $x$ represents an approximation to the expected value as a function of $x$ under the distribution enforced by $\bar{y}$.

\subsection{A convergent algorithm}
\label{subsec:alg_convergence}
With the elements defined in \Cref{subsec:opt_cut} we can now introduce a complete algorithm, which we refer to as the Benders-based method (BBM). The pseudocode is presented in Algorithm \ref{alg:bsc}. In this algorithm, we refer to RMP$^{k}$ as the instance of the RMP solved at iteration $k$, which may include multiple optimality cuts \eqref{eq:single_cut:v1}. The value $v_{MP}^{Best}$ holds the objective function value of the best primal solution the algorithm has found, while $v_{MP}^{k}$ holds the objective function value of the primal solution found at iteration $k$. We note that the values $(\bar{x},\bar{y})$ of a solution to RMP also satisfy the constraints of MP. Thus, constructing a solution to MP from a solution to the RMP only requires evaluating $Q_{p}(\bar{x},\bar{y})$ for each product $p \in \mathcal{P}$.

\begin{algorithm}[htp]
\caption{Benders-based algorithm}\label{alg:bsc}
\begin{algorithmic}
\small
\Require Optimality tolerance $\epsilon$
\Require Time limit $\tau$
\State Set $v_{MP}^{Best}=0$
\State Instantiate RMP$^{0}$ with bound \eqref{eq:bnd} $\forall p \in \mathcal{P}$
\State \texttt{STOP}$\gets$ \texttt{FALSE}
\While{NOT \texttt{STOP}}
	\State Solve RMP$^{k}$ for solution $(\bar{x}^{k},\bar{y}^{k},\bar{\mu}^{k})$ and bound $v_{RMP}^{k}$
	\State Set $v_{MP}^{k} = - \sum_{f \in \mathcal{F}} C_{pf}\bar{x}^k_{pf} $
	\For{$p \in \mathcal{P}$}
        \State Compute $Q_p(\bar{x},\bar{y})$  
		\If{$\bar{\mu}_p>Q_p(\bar{x},\bar{y})$}
			\State Save cut \eqref{eq:single_cut:v1} for $p$
		\EndIf
		\State Set $v_{MP}^{k} = v_{MP}^{k} + Q_{p}(\bar{x},\bar{y})$
	\EndFor
    \If{$v_{MP}^{k} > v_{MP}^{Best}$}
		\State Set $v_{MP}^{Best} = v_{MP}^{k}$
    \EndIf
	\If{No cuts saved}
		\State Set $v_{MP}^{Best} = v_{MP}^{k}$
		\State \texttt{STOP}$\gets$ \texttt{TRUE}
	\Else	
		\State Compute  $gap=(v_{RMP}^{k} - v_{MP}^{Best})/v_{RMP}^{k}$
		\If{$gap \leq \epsilon$ OR \texttt{elapsed\_time} $>\tau$}
			\State \texttt{STOP}$\gets$ \texttt{TRUE}
		\Else
			\State Instantiate RMP$^{k+1}$ with cuts \eqref{eq:single_cut:v1} in RMP$^{k}$
			\State Add to RMP$^{k+1}$ all saved cuts
		\EndIf
	\EndIf
\EndWhile
\end{algorithmic}
\end{algorithm}

The following proposition clarifies that the algorithm terminates in at most a finite number of iterations, even when $\epsilon=0$ and $\tau=\infty$, where $\epsilon$ is the target optimality gap and $\tau$ is the allowed computation time.

\begin{proposition}\label{prop:convergence}
Assume all optimality cuts \eqref{eq:single_cut:v1} have been added to RMP at iteration $k$. Then, for solution $(x^k,y^k,\mu^k)$ we have that
$$\mu^k=Q_p(x^k,y^k)$$
for all $p\in\mtc{P}$.
\end{proposition}
\proof{Proof of Proposition \ref{prop:convergence}}
This can be proven by contradiction. Assume all $\sum_{p\in\mtc{P}}\sum_{d\in\mtc{D}_p}2^{|\mtc{S}_{pd}|}$ optimality cuts \eqref{eq:single_cut:v1} have been added to RMP and that for $(x^k,y^k,\mu^k)$ we have that
$\mu^k > Q_p(x^k,y^k)$ 
for some $p\in\mtc{P}$. Then solution $(x^k,y^k,\mu^k)$ is cut off by 
\begin{align*}
\mu_p \leq & \sum_{s \in \mathcal{S}^{\downarrow}_{pd(y^k)}(x^k,y^k)} \pi_{sd(y^k)} \bigg(P_p \sum_{f \in \mathcal{F}}Y_{pfd(y^k)s}x_{pf} \bigg) \\ \nonumber 
& +\sum_{s \in \mathcal{S}^{\uparrow}_{pd(y^k)}(x^k,y^k)} \pi_{sd(y^k)}\bigg(O_p\sum_{f\in\mathcal{F}}Y_{pfd(y^k)s}x_{pf}+(P_p-O_p)D_{pd(y^k)s}\bigg)\\
    \nonumber  &+ M_{p}\bigg(\vert \mathcal{F} \vert - \sum_{f \in \mathcal{F}} y_{p,f,l(p,f,d(y^k))}\bigg) 
\end{align*}
as shown in Proposition \ref{prop:oc:cut}. This contradicts the assumption that all cuts are in RMP and completes the proof. 
\endproof

\subsection{Valid inequalities}
\label{subsec:pos_overage_bound}
We next present two further valid inequalities that can be added to RMP in an \textit{a priori} manner so that it better approximates the potential expected revenues earned from each product and is a tighter relaxation. 

For the first, we let $\bar{Y}_{pf}=\max_{d\in\mathcal{D}_p,s\in\mathcal{S}_{pd}}\{Y_{pfds}\}$
denote the largest yield for product $p$ at facility $f$ over all distributions. 
Given a solution  $(x^*,y^*)$ to MP, we have that $Q_p(x^*,y^*)\leq P_p\sum_{f\in\mathcal{F}}\bar{Y}_{pf}x_{pf}^{*}$ as the expected revenues from a given product
cannot exceed those realized when the highest yield is achieved at each facility and all finished goods inventory is sold at full price. Therefore, the following inequality is valid and can be added to RMP.
\begin{equation}\label{eq:vi1}
    \mu_p\leq P_p\sum_{f\in\mathcal{F}}\bar{Y}_{pf}x_{pf}\tag{VI1}
\end{equation}

The next valid inequality relies on the quantity $Y^E_{pfd} =  \sum_{s\in\mathcal{S}_{pd}}\pi_{sd}Y_{pfds}$ that represents the expected yield of product $p$ when allocated to facility $f$ under distribution $d.$ Given that definition, we next prove that the following inequality is valid.
\begin{equation}\label{eq:vi2}
\mu_p \leq P_p\sum_{f\in\mathcal{F}}\max_{d\in\mathcal{D}_{p}} \big\{Y^E_{pfd}\big\}x_{pf} \tag{VI2}   
\end{equation}

\noindent
\begin{proposition}\label{prop:VI2}
The inequality $$Q_p(x,y)\leq P_p\sum_{f\in\mathcal{F}}\Bigg[\max_{d\in\mathcal{D}_p}\{Y^E_{pfd}\big\}x_{pf}\Bigg]$$ is satisfied by all feasible solutions $(x,y)$ to MP.
\end{proposition}
\noindent
\begin{proof}{Proof of Proposition \ref{prop:VI2}}
Consider a solution $(\tilde{x},\tilde{y})$ to MP. As before, we let $Q_{pds}(\tilde{x},\tilde{y})$ denote the expected revenues earned for product $p$ under distribution $d \in \mathcal{D}_p$ and scenario $s\in\mathcal{S}_{pd}$ given the allocations prescribed by $\tilde{x}.$ We have that $Q_{pds}(\tilde{x},\tilde{y}) \leq P_p\sum_{f\in\mathcal{F}}Y_{pfds}\tilde{x}_{pf}$ as the most revenue that can be earned is when all finished goods are sold at full price.

Taking expectations with respect to the scenarios in $\mathcal{S}_{pd}$ we have
\begin{align*}Q_{pd}(\tilde{x},\tilde{y})\leq & P_p\sum_{f\in\mathcal{F}}\sum_{s\in\mathcal{S}_{pd}}\big(\pi_{sd}Y_{pfds}\tilde{x}_{pf}\big)\\
                             &= P_p\sum_{f\in\mathcal{F}}Y^E_{pfd}\tilde{x}_{pf}
\end{align*}
where $Q_{pd}(x,y)$ represents the expected revenue for product $p$ under distribution $d$.
Furthermore, we have that
\begin{align*}\max_{d\in\mathcal{D}_{p}}Q_{pd}(\tilde{x},\tilde{y})\leq & \max_{d\in\mathcal{D}_{p}} \big(P_p\sum_{f\in\mathcal{F}}Y^E_{pfd}\tilde{x}_{pf}\big)\\
                                    & = P_p\sum_{f\in\mathcal{F}}\max_{d\in\mathcal{D}_{p}} \big(Y^E_{pfd}\big)\tilde{x}_{pf}.
\end{align*}
Thus, we can conclude that  $$Q_{p}(\tilde{x},\tilde{y}) \leq \max_{d\in\mathcal{D}_{p}}Q_{pd}(\tilde{x},\tilde{y})\leq P_p\sum_{f\in\mathcal{F}}\max_{d\in\mathcal{D}_{p}} \{Y^E_{pfd}\}\tilde{x}_{pf}$$
as required .
\end{proof}

\section{Computational analysis}
\label{sec:comp_analysis}
\noindent
In this section we computationally analyze two issues. First, we estimate the value in modeling uncertainty in both production yields and product demands by computing the well-known \textit{Value of Stochastic Solution} (VSS), and under different sources of uncertainty. Specifically, we compute the VSS when there is uncertainty only in production yields, only in product demands, and uncertainty in both. We next study the effectiveness of the proposed Benders-based solution method. However, to ground those analyses we first discuss the instances and computational environment used in the computational study. 

\subsection{Instances}
\label{subsec:instances}
\noindent
In this section we  discuss the instances used in the experiments that form the basis of our compuational analyses.  We first note that all instances underlying our computational study are randomly generated. Thus, we will next describe the parameter values used when generating those instances. We recall that there are two primary parameters that define the context in which the production planning problem considered in this paper would be formulated and solved. Specifically, the number of products, $\vert \mathcal{P} \vert$, to be manufactured and the number of facilities, $\vert \mathcal{F} \vert$, in which they can be manufactured. A third parameter defining an instance is the number of production levels, $\vert \mathcal{L}_{pf} \vert$, for product $p$ at facility $f$ for which there is a different statistical distribution describing the production yield. All instances considered in our experiments contained a number of products that ranged between five and ten and a number of facilities that ranged between two and five. We note the instances were constructed such that each product could be produced in each facility. In addition, all products had either two or three production levels at each facility.  Thus, we define a class of instance by the triplet $(\vert \mathcal{F} \vert, \vert \mathcal{P} \vert, \vert \mathcal{L} \vert)$, wherein $\vert \mathcal{L}_{pf} \vert = \vert \mathcal{L} \vert \;\; \forall f \in \mathcal{F}, p \in \mathcal{P}.$

For an instance within a given class we randomly determined the remaining parameter values. We determined the value of the capacity, $B_{f}$, of facility $f$ via the following approach. We first determined total expected demand of each product over all scenarios for that instance. More specifically, let $\nu_{p}$ denote the expected demand of product $p.$ We then computed the value $\chi = \sum_{p \in \mathcal{P}} \nu_{p} / \vert \mathcal{F} \vert,$ which represents the expected capacity needed at each facility if demands were met and allocated evenly across facilities. Next, for each facility $f$ we drew its capacity from the uniform distribution $[\chi - \gamma\chi, \chi + \gamma\chi]$ with $\gamma = .1$

Regarding costs and revenues, the cost $C_{fp}$ of manufacturing product $p \in \mathcal{P}$ in facility $f \in \mathcal{F}$ was randomly drawn from the range $[60,80].$ The price $P_{p}$ at which product $p \in \mathcal{P}$ is sold was randomly drawn from the range $[125,185]$. As such, the expected gross margin on products sold is 54\%.  Finally, the revenue $O_{p}$ recovered from left-over inventory of product $p \in \mathcal{P}$ was randomly drawn from the range $[15,40].$ Note this ensures that in each instance there is a loss on products manufactured but not sold at full price.

As noted, we generated instances with either two or three production levels for each product at each facility. Given a number of production levels, the same procedure is used to generate the bounds for each level for each product and each facility. When there are two levels, the bounds for the levels are $(L_{pf0},U_{pf0}) = (0,.75\nu_{p})$ and $(L_{pf1},U_{pf1}) = (.75\nu_{p},\nu_{p}).$  When there are three levels, the bounds are $(L_{pf0},U_{pf0}) = (0,.5\nu_{p}), (L_{pf1},U_{pf1}) = (.50\nu_{p},.75\nu_{p}), $ and $(L_{pf2},U_{pf2}) = (.75\nu_{p},\nu_{p})$.

We next discuss the scenario generation process. We first recall that the model and the solution method presented allow the treatment of endogenous demand distributions. However, we are unaware of sufficient justification to model this relationship. Thus, in the scenarios generated demand is treated as a fully exogenous random variable. This is done by appending the same marginal demand distribution to the joint distribution of the yield.

Thus, associated with each product in an instance are $\vert \mathcal{D}_{p} \vert =  \vert  \mathcal{L} \vert^{\vert \mathcal{F} \vert} + 1$ distributions, wherein the first $ \vert \mathcal{L} \vert^{\vert \mathcal{F} \vert} $ are distributions of production yields at different facilities for different sets of production levels and the remaining is of product demands. The scenario generation process used to generate instances presumed a given number of scenarios to represent each distribution. Specifically, $S =$ 5,10,15,20, and 25 scenarios per distribution were considered. As a result, an instance consists of $S\vert \mathcal{P} \vert (  \vert \mathcal{L} \vert^{\vert \mathcal{F} \vert} + 1)$ scenarios, which can range from 105 to 60,750 scenarios. Regarding demands, we presumed the distribution of demand for each product was normal with mean 20,000 and standard deviation of 15,000. 

For production yields, we considered normal distributions albeit with different parameters. We note that for a given product and facility we presume the mean of the yield distribution is the same for all production levels. However, for a given mean, the larger the production level the smaller the standard deviation. In other words, we model that variability in yields decreases as the level of production increases. We summarize the different yield distributions we considered in Table \ref{tbl:yield_dists}. When generating an instance with a given number of production levels, one of the relevant distributions in Table \ref{tbl:yield_dists} was randomly assigned to each product and facility. Then, to generate scenarios, values were randomly drawn from these distributions. However, we note that we truncated the values at $.25$ and $1.$

\begin{table}[htp]
\small
\center
\begin{tabular}{|c|c|c|c|}
\hline
\multicolumn{4}{|c|}{Distributions for $\vert \mathcal{L} \vert = 2$} \\
\hline
Distribution & Level & Mean & Standard deviation \\
\hline
\multirow{2}{*}{1} & 1 & 0.50 & 0.20 \\
 & 2 & 0.50 & 0.10 \\
\multirow{2}{*}{2} & 1 & 0.70 & 0.20 \\
 & 2 & 0.70 & 0.10 \\
\multirow{2}{*}{3} & 1 & 0.90 & 0.05 \\
 & 2 & 0.90 & 0.01 \\
\hline
\multicolumn{4}{|c|}{Distributions for $\vert \mathcal{L} \vert = 3$} \\
\hline
Distribution & Level & Mean & Standard deviation \\
\hline
\multirow{3}{*}{1} & 1 & 0.50 & 0.20 \\
 & 2 & 0.50 & 0.15 \\
 & 3 & 0.50 & 0.10 \\
\multirow{3}{*}{2} & 1 & 0.70 & 0.20 \\
 & 2 & 0.70 & 0.15 \\
 & 3 & 0.70 & 0.10 \\
\multirow{3}{*}{3} & 1 & 0.90 & 0.05 \\
 & 2 & 0.90 & 0.03 \\
 & 3 & 0.90 & 0.01 \\
\hline
\end{tabular}
\caption{Production yield distributions, two and three production levels}
\label{tbl:yield_dists}
\end{table}

For a given class and given number of scenarios per distribution we generated five instances. In total, we considered four different values for $\vert \mathcal{F} \vert$, six different values for $\vert \mathcal{P} \vert$ and two different values of $\vert \mathcal{L} \vert$, yielding 48 different classes of instances. Then, for each class we considered five different numbers of scenarios per distribution, yielding 240 instance configurations. With five instances per configuration, our experimental study is based upon 1,200 instances in total.

\subsection{Computational environment}
\label{subsec:comp_environment}

In this section, we describe the computational environment in which all experiments were run. All code is implemented in Java. All optimization models solved by that code were done so using CPLEX 20.1. We note that the method proposed in \Cref{sec:benders} has been embedded in a branch-and-cut framework in which optimality cuts are separated and added to the RMP at nodes at which integer solutions are discovered.

All tests were performed on computing nodes equipped with 40 CPUs and 188 GB memory. However, no parallelization techniques were used when executing the proposed Benders-based algorithm. Each of those experiments was run using a single thread. However, when solving extensive linearized forms with CPLEX, CPLEX was configured to use its default deterministic parallel search strategy for the Branch \& Bound method, which used up to $32$ threads. Both executing the proposed method and solving the deterministic equivalent were done with an optimality tolerance of $0.0001$ and a time limit of $1,800$ seconds.

\subsection{Value of modeling uncertainty}
\label{subsec:value_uncertain}
We next study the value in modeling uncertainty in both supply and demand in this context. To do so, we compute the \textit{Value of Stochastic Solution} ($VSS$) (see \cite{Bir82} and its extensions \cite{EscGMP07,MagW12,PanB20}) adapted to problems with endogenous uncertainty. Generally speaking, the $VSS$ represents the relative difference in objective function values between the optimal solution to the stochastic program and a solution derived by optimizing with respect to expected values. Particularly, when solving the expected value problem, random parameters are replaced by their expected values in each possible distribution and the choice of production levels no longer determines a probability distribution but simply expected yield and demand realizations. The solution to the expected value problem (which determines a probability distribution of yield and demand) is then used to formulate and solve instances of the second-stage problem for each scenario of the distribution enforced. Observe that the distribution enforced may be different from that enforced by the optimal solution to the stochastic program. The solutions to the second-stage problems, along with the first stage values from the expected value problem, are then used to construct a complete solution to the stochastic program. With $v_{SP}$ representing the objective function value of an optimal solution to the stochastic program and $v_{EV}$ the objective function value of the expected value solution, the $VSS$ can be computed as $(v_{SP} - v_{EV})/v_{SP}.$

In our problem, there is uncertainty with respect to both supply and demand. As a result, we compute multiple $VSS$s. The first, which we label $VSS_{supply}$, measures the value in modeling uncertainty in supply, presuming demand is uncertain. To do so, the expected value solution is constructed by solving the stochastic program wherein yield values are not treated as uncertain and instead set to their expected values. Like the traditional $VSS,$ the resulting values of first stage decision variables are used to formulate and solve second stage problems for each scenario. Letting $v_{EV}^{supply}$ represent the objective function value of the resulting solution we have $VSS_{supply} = (v_{SP} - v_{EV}^{supply})/v_{SP}.$ The second, which we label $VSS_{demand}$, measures the value in modeling uncertainty in demand, presuming yields are uncertain. In this case, the expected value solution is constructed by solving the stochastic program wherein demand values are set to their expected values. Letting $v_{EV}^{demand}$ represent the objective function value of the resulting solution we have $VSS_{demand} = (v_{SP} - v_{EV}^{demand})/v_{SP}.$ Finally, we denote the classical $VSS$, which is computed by constructing an expected value solution based on expectations with respect to both demands and production yields, by $VSS_{full}.$ 

Thus, in the next three tables we report the $VSS_{supply}, VSS_{demand}, $ and $VSS_{full}$ statistics, averaged over different sets of instances. We note we only report results for instances which could be solved to optimality. In Table \ref{tbl:vss_products} we report each of the $VSS$ statistics, albeit averaged over instances with the same number of products. Table \ref{tbl:vss_facilities} is similar, albeit averaged over instances with the same number of facilities. Finally, Table \ref{tbl:vss_levels} reports $VSS$ statistics averaged over instances with the same number of production levels. 

\begin{table}[htp]
\center
\small
\begin{tabular}{|c|c|c|c|c|c|c|c|}
\hline
& \multicolumn{6}{c}{$\vert \mathcal{P} \vert$} & \\
$VSS$ & 5 & 6 & 7 & 8 & 9 & 10 & Average \\
\hline
$VSS_{supply}$ & 1.14\% & 0.89\% & 1.03\% & 1.02\% & 0.90\% & 0.92\% & 0.99\% \\
$VSS_{demand}$ & 7.83\% & 5.13\% & 5.16\% & 5.34\% & 3.97\% & 4.09\% & 5.37\% \\
$VSS_{full}$ & 22.03\% & 17.02\% & 18.53\% & 17.92\% & 13.48\% & 15.43\% & 17.64\% \\
\hline
\end{tabular}
\caption{$VSS$ statistics by number of products}
\label{tbl:vss_products}
\end{table}

\begin{table}[htp]
\begin{subtable}[c]{.5\textwidth}
\small
\centering
\begin{tabular}{|c|c|c|c|c|}
\hline
& \multicolumn{4}{c|}{$\vert \mathcal{F} \vert$}  \\
$VSS$ & 2 & 3 & 4 & 5 \\
\hline
$VSS_{supply}$ & 0.86\% & 1.01\% & 1.07\% & 1.04\% \\
$VSS_{demand}$ & 4.83\% & 4.92\% & 6.37\% & 5.59\% \\
$VSS_{full}$ & 15.49\% & 17.22\% & 19.74\% & 18.76\% \\
\hline
\end{tabular}
\subcaption{By number of facilities}
\label{tbl:vss_facilities}
\end{subtable}
\begin{subtable}[c]{.5\textwidth}
\center
\small
\begin{tabular}{|c|c|c|}
\hline
& \multicolumn{2}{c|}{$\vert \mathcal{L} \vert$}  \\
 & 2 & 3 \\
 \hline
$VSS_{supply}$ & 1.19\% & 0.71\% \\
$VSS_{demand}$ & 5.73\% & 4.89\% \\
$VSS_{full}$ & 19.44\% & 15.24\% \\
\hline
\end{tabular}
\subcaption{By number of production levels}
\label{tbl:vss_levels}
\end{subtable}
\caption{$VSS$ statistics}
\end{table}
We see from these tables that taken in isolation, modeling uncertainty in demands leads to a greater difference in solution quality than modeling uncertainty in production yields. However, we also see that modeling uncertainty in both leads to the greatest improvement in solution quality. Finally, we do not conclude from these tables that there is a correlation between the parameter values that define a class of instance and any of the $VSS$ statistics.

\subsection{Performance of proposed solution approach}
\label{subsec:perf_benders}
\noindent
In this section, we study the performance of the Benders-based solution approach presented in Section \ref{sec:benders}. We note that to generate the results presented in this section, computational experiments were run with a time limit of 1,800 seconds and an optimality tolerance of .0001 (the default for CPLEX). We focus our comparisons on two performance metrics. The first is the percentage of the 1,200 instances that a method could solve within the time limit and to the desired optimality tolerance. The second is the average time in seconds the method took to solve an instance to that tolerance, averaged over instances the method was able to solve.

We first compare the computational performance of the proposed Benders-based approach with that of two benchmarks. The first is a commercial solver, CPLEX, solving the full instance of the PP-DESUP, after linearization. We label this benchmark CPLEX-FULL. Given that the second stage variables in PP-DESUP are continuous, the problem is a candidate for a classical Benders decomposition-based approach. To benchmark against such an approach, we use the automated Benders decomposition of CPLEX, which we label as CPLEX-BD. We first consider the performance of the approach presented in Section \ref{subsec:alg_convergence}, which we label BSC as it involves a single optimality cut for each product at each iteration. We then study the impact of augmenting BBM with the valid inequalities VI1 and VI2. We label such algorithm configurations as BBM + VI1 and BBM+VI2 respectively. 

We report in Table \ref{tbl:comp_bench} the values of the two performance metrics for the proposed Benders-based method and the two benchmarks. We see that the Benders-based approach was able to solve many more instances than either of the two benchmarks. In addition, even though it could solve many more instances, it was also able to solve that larger set of instances in much less time. Thus, we conclude from the results in this table that the Benders-based approach is computationally superior to the two benchmarks under consideration.

\begin{table}[htp]
\center
\begin{tabular}{|c|c|c|}
\hline
Method & \% solved & Time to solve (sec.) \\
\hline
CPLEX-Full & 23.98\% & 855.28 \\
CPLEX-BD & 8.82\% & 1151.82 \\
BBM & 71.76\% & 151.27 \\
\hline
\end{tabular}
\caption{Comparison of Benders-based algorithm with benchmarks}
\label{tbl:comp_bench}
\end{table}

We next study the impact of augmenting the Benders-based method with the two valid inequalities. We focus on the same two performance metrics as reported in Table \ref{tbl:comp_bench}, but averaged over instances based on the same number of facilities (Tables \ref{tbl:bd_pct_solved_byf} and \ref{tbl:bd_solve_time_byf}), products (Tables \ref{tbl:bd_pct_solved_byp} and \ref{tbl:bd_solve_time_byp}), or production levels (Tables \ref{tbl:bd_pct_solved_byl} and \ref{tbl:bd_solve_time_byl}). To further understand the performance of the two benchmarks, we also report values of the performance metrics for those methods.

\begin{table}[htp]
\begin{subtable}[c]{.5\textwidth}
\tiny
\centering
\begin{tabular}{|c|c|c|c|c|c|}
\hline
& \multicolumn{4}{c|}{$\vert \mathcal{F} \vert$} &  \\
 Method & 2 & 3 & 4 & 5 & Average\\
 \hline
CPLEX-Full & 51.01\% & 17.03\% & 10.82\% & 7.18\% & 23.98\% \\
CPLEX-BD & 23.34\% & 8.90\% & 2.18\% & 0.00\% & 8.82\% \\
BBM & 94.93\% & 76.11\% & 64.51\% & 51.52\% & 71.76\% \\
BBM+VI1 & 95.97\% & 76.95\% & 65.89\% & 55.18\% & 73.47\% \\
BBM+VI2 & 100.00\% & 97.27\% & 82.21\% & 69.67\% & 87.24\% \\
\hline
\end{tabular}
\subcaption{\% solved}
\label{tbl:bd_pct_solved_byf}
\end{subtable}
\begin{subtable}[c]{.5\textwidth}
\center
\tiny
\begin{tabular}{|c|c|c|c|c|c|}
\hline
& \multicolumn{4}{c|}{$\vert \mathcal{F} \vert$}  & \\
 Method & 2 & 3 & 4 & 5 & Average\\
 \hline
CPLEX-Full &  762.01  &  946.56  &  1,178.21  &  987.67  &  855.28  \\
CPLEX-BD &  1,161.05  &  1,111.79  &  1,222.12  & N/A &  1,151.82  \\
BBM &  89.76  &  102.82  &  198.27  &  276.81  &  151.27  \\
BBM+VI1 &  94.11  &  122.70  &  197.27  &  291.81  &  162.04  \\
BBM+VI2 &  27.07  &  106.48  &  147.56  &  123.41  &  96.65  \\
\hline
\end{tabular}
\subcaption{Solve time}
\label{tbl:bd_solve_time_byf}
\end{subtable}
\caption{Performance metrics by number of facilities}
\end{table}

\begin{table}[htp]
\begin{subtable}[c]{.95\textwidth}
\center
\tiny
\begin{tabular}{|c|c|c|c|c|c|c|c|}
\hline
& \multicolumn{6}{c|}{$\vert \mathcal{P} \vert$} &  \\
Method & 5 & 6 & 7 & 8 & 9 & 10 & Average \\
\hline
CPLEX-Full & 51.72\% & 39.88\% & 25.47\% & 13.46\% & 6.10\% & 4.35\% & 23.98\% \\
CPLEX-BD & 28.93\% & 14.52\% & 5.32\% & 2.79\% & 0.00\% & 0.00\% & 8.82\% \\
BBM & 100.00\% & 92.35\% & 77.44\% & 64.97\% & 60.10\% & 35.71\% & 71.76\% \\
BBM+VI1 & 100.00\% & 95.98\% & 80.40\% & 65.66\% & 59.50\% & 39.70\% & 73.47\% \\
BBM+VI2 & 100.00\% & 100.00\% & 91.37\% & 81.00\% & 77.89\% & 73.50\% & 87.24\% \\
\hline
\end{tabular}
\subcaption{\% solved}
\label{tbl:bd_pct_solved_byp}
\end{subtable}
\begin{subtable}[c]{.95\textwidth}
\center
\tiny
\begin{tabular}{|c|c|c|c|c|c|c|c|}
\hline
& \multicolumn{6}{c|}{$\vert \mathcal{P} \vert$}  & \\
 Method & 5 & 6 & 7 & 8 & 9 & 10 & Average \\
 \hline
CPLEX-Full &  779.02  &  851.63  &  892.30  &  938.38  &  992.14  &  1,209.10  &  855.28  \\
CPLEX-BD &  1,144.34  &  1,166.73  &  957.77  &  1,544.62  & N/A& N/A &  1,151.82  \\
BBM &  51.62  &  112.49  &  130.79  &  105.21  &  304.62  &  399.72  &  151.27  \\
BBM+VI1 &  44.64  &  143.89  &  178.39  &  129.29  &  258.71  &  372.32  &  162.04  \\
BBM+VI2 &  25.07  &  65.25  &  132.60  &  111.55  &  132.47  &  136.61  &  96.65  \\
\hline
\end{tabular}
\subcaption{Solve time}
\label{tbl:bd_solve_time_byp}
\end{subtable}
\caption{Performance metrics of methods by number of products}
\end{table}

\begin{table}[htp]
\begin{subtable}[c]{.5\textwidth}
\tiny
\centering
\begin{tabular}{|c|c|c|c|}
\hline
& \multicolumn{2}{c|}{$\vert \mathcal{L}_{pf} \vert$} &  \\
 Method & 2 & 3  & Average\\
 \hline
CPLEX-Full & 32.98\% & 12.06\% & 23.98\% \\
CPLEX-BD & 15.17\% & 2.03\% & 8.82\% \\
BBM & 92.98\% & 51.01\% & 71.82\% \\
BBM+VI1 & 93.11\% & 53.86\% & 73.47\% \\
BBM+VI2 & 100.00\% & 74.62\% & 87.24\% \\
\hline
\end{tabular}
\subcaption{\% solved}
\label{tbl:bd_pct_solved_byl}
\end{subtable}
\begin{subtable}[c]{.5\textwidth}
\center
\tiny
\begin{tabular}{|c|c|c|c|}
\hline
& \multicolumn{2}{c|}{$\vert \mathcal{L}_{pf} \vert$} &  \\
 Method & 2 & 3  & Average\\
 \hline
 CPLEX-Full &  850.64  &  872.13  &  855.28  \\
CPLEX-BD &  1,135.83  &  1,279.73  &  1,151.82  \\
BBM & 91.79  &  257.87  &  151.27  \\
 BBM+VI1 & 85.58  &  293.99  &  162.04  \\
 BBM+VI2 &11.33  &  209.63  &  96.65  \\
 \hline
 Average & 61.55  &  248.56  &  134.12  \\
\hline
\end{tabular}
\subcaption{Solve time}
\label{tbl:bd_solve_time_byl}
\end{subtable}
\caption{Performance metrics of Benders-based method by number of production levels}
\end{table}

We first note that CPLEX-BD performs worse (i.e. solves fewer instances, requires more time) than CPLEX-Full for every value of $\vert \mathcal{P} \vert, \vert \mathcal{F} \vert, $ and $\vert \mathcal{L}_{pf} \vert.$ We next observe that the first valid inequality (VI1) enables the Benders-based method to solve more instances. While the average solve time increases, this is due to BBM+VI1 solving instances that BBM could not solve. We note that BBM+VI1 was able to solve every instance that BBM could solve and it solved those instances in less time, 129.95 seconds.

While the first valid inequality improves the performance of the Benders-based method, the second valid inequality (VI2) does so much more. Further, for every value of $\vert \mathcal{P} \vert, \vert \mathcal{F} \vert, $ or $\vert \mathcal{L}_{pf} \vert$ we see that using VI2 enabled the Benders-based method to solve at least as many instances and in much less time. Next, we note that an increase in any of the parameter values lead to a decrease in performance (fewer instances solved, more time required to do so) for all methods under consideration. However, increasing the number of production levels, arguably leads to the greatest degradation in performance. This is likely due to a larger number of production levels leading to a much larger set of scenarios, $S,$ overall. 

We also note that on average, when BBM could not solve an instance, the average optimality gap it reported for those instances was 134.04\%. With the use of VI1, that average drops to 104.19\%. With the use of VI2 that gap drops further, to 30.48\%. Finally, we observe that BBM, with and without valid inequalities, never required more than a few megabytes of memory during its execution.

To understand why VI2 has such a dramatic impact on the performance of the Benders-based method, we recall that VI2 is added to the RMP \textit{a priori}. Therefore, we next study the relative gap between the upper bound on the optimal objective function value, $v_{MP}^{*},$ produced at the first iteration by solving the RMP both when VI2 is added to and when it is not. Namely, we let $v_{RMP}^{1}$ denote the bound produced by solving RMP at the first iteration of the Benders-based method, before any optimality cuts have been added. We let $v_{RMP}^{1-VI2}$ represent the same quantity, only when RMP is strengthened with VI2. Given these values, we compute the gap $RMP_{gap} = (v_{RMP}^{1}-v_{RMP}^{1-VI2})/v_{RMP}^{1}.$We report averages of these bounds and this gap, averaged over instances based on the same number of products in Table \ref{tbl:avg_rmp_gap_by_p}. Table \ref{tbl:avg_rmp_gap_by_f} is similar, only averaged over instances based on the same number of facilities. 

\begin{table}[htp]
\begin{subtable}[c]{.5\textwidth}
\small
\centering
\begin{tabular}{|c|c|c|c|}
\hline
$\vert \mathcal{P} \vert$ & $v_{RMP}^{1}$ & $v_{RMP}^{1-VI2}$ & $RMP_{gap}$ \\
\hline
5 &  10,061,828.38  &  1,933,102.16  & -420.50\% \\
6 &  13,901,803.48  &  2,862,069.48  & -385.73\% \\
7 &  17,792,070.22  &  3,718,940.70  & -378.42\% \\
8 &  19,232,664.61  &  3,934,961.43  & -388.76\% \\
9 &  26,841,465.65  &  6,018,151.87  & -346.01\% \\
10 &  28,724,420.99  &  6,316,784.74  & -354.73\% \\
\hline
\end{tabular}
\subcaption{By number of products}
\label{tbl:avg_rmp_gap_by_p}
\end{subtable}
\begin{subtable}[c]{.5\textwidth}
\center
\small
\begin{tabular}{|c|c|c|c|}
\hline
$\vert \mathcal{F} \vert$ & $v_{RMP}^{1}$ & $v_{RMP}^{1-VI2}$ & $RMP_{gap}$ \\
\hline
2 &  16,285,038.80  &  3,236,958.38  & -403.10\% \\
3 &  19,261,859.83  &  3,936,760.71  & -389.28\% \\
4 &  20,629,325.55  &  4,464,340.33  & -362.09\% \\
5 &  21,646,685.17  &  4,916,712.95  & -340.27\% \\
\hline
\end{tabular}
\subcaption{By number of facilities}
\label{tbl:avg_rmp_gap_by_f}
\end{subtable}
\caption{Improvement in bound due to VI2}
\end{table}

We see that adding VI2 to RMP leads to a significant reduction in the bound produced by solving RMP before the addition of any optimality cuts. However, this reduction tends to decrease as instance parameter values increase, particularly the number of facilities. We note that VI1 also improves the bound at the first iteration, only less so. Averaged over all instances, using inequality VI1 reduces the bound at the first iteration by 82.25\%. 
 
\section{Conclusion and future work}
\label{sec:conclusion_future}
\noindent
In this paper, we studied a production planning problem in which there is uncertainty in both supply and demand. Further complicating matters, uncertainty in supply is endogeneous as it depends on choices regarding production levels of products at facilities. We formulated the problem as a two-stage stochastic program in which scenarios contain both production yield and customer demand information. We proposed a Benders-based solution approach wherein the objective function of the relaxed master problem includes decision variables that approximate the expected profits for a product given production choices made in the first stage. In the proposed scheme, when the approximation variables over-estimate expected profits an optimality cut is generated and added to the relaxed master problem. The solution process then repeats. We proved there are a finite number of such optimality cuts and thus  that the algorithm is guaranteed to converge in a finite number of iterations. We also presented additional valid inequalities to strengthen the relaxed master problem. 

With an extensive computational study we first analyzed the value in modeling supply and demand uncertainty separately as well as together. We saw that modeling either supply or demand unertainty lead to an improvement in solution quality in a stochastic setting. However, when modeling both the improvement was much more pronounced. Finally, we demonstrated that the proposed Benders-based scheme outperforms two benchmarks. The first benchmark is an off-the-shelf solver solving the full stochastic program. As the second stage variables in the proposed stochastic program are continuous, the second benchmark is a standard implementation of Benders, as implemented by an off-the-shelf solver. 

We see multiple avenues for future work in this area. The first is managerial. In this paper we focused on proposing an efficient, exact, algorithm for the problem under consideration. However, such a problem arises in multiple industrial contexts, including agriculture and manufacturing. An extensive analysis of high-quality solutions to the problem could yield tactics for managing uncertainty in both supply and demand, particularly when supply uncertainty is endogenous, for those contexts. 

The second is algorithmic and is motivated by assumptions underlying the proposed work that may not hold in certain practical applications. For example, a critical assumption of the model and Benders-based scheme we propose is that the yields across products are independent. However, in many practical contexts this is not the case. Relatedly, the proposed Benders-based scheme relies on the fact that the determination of expected profits is separable by product. However, in cases wherein one product may be substituted for another, such separability is lost. Also, the inequality generated by the proposed Benders-based scheme is specific to the distribution induced by the master problem solution at hand. The validity of that inequality requires an expression that indicates whether a master problem solution induces that same distribution. There may be opportunities to strengthen the cut with a different expression that captures this implication.

The third relates to implementation of the proposed model in industrial contexts as it presupposes the fitting of multiple distributions that describe yields of a product at different facilities. On the one hand, we see opportunities for research into how to best determine and fit those distributions. On the other hand, we see opportunities to develop robust optimization-type models which require less information regarding distributions.

Finally, we recognize the need for additional research to estimate the value of modeling endogenous uncertainty. To provide such estimate tests and metrics would be required that compare the solution to a stochastic program with endogenous uncertainty to that of a benchmark stochastic program with exogenous uncertainty. Nevertheless, the decisions
of how to choose the exogenous distribution for the stochastic program with exogenous uncertainty lends itself to different interpretations and is potentially problem-specific.

\bibliographystyle{abbrv}
\bibliography{preprint.bib}

\begin{thebibliography}{10}

\bibitem{Ahm00}
S.~Ahmed.
\newblock {\em {Strategic planning under uncertainty: Stochastic integer
  programming approaches}}.
\newblock Phd, University of Illinois at Urbana-Champaign, 2000.

\bibitem{ahumada2009application}
O.~Ahumada and J.~R. Villalobos.
\newblock Application of planning models in the agri-food supply chain: A
  review.
\newblock {\em European journal of Operational research}, 196(1):1--20, 2009.

\bibitem{AlkF83}
F.~A. Al-Khayyal and J.~E. Falk.
\newblock Jointly constrained biconvex programming.
\newblock {\em Mathematics of Operations Research}, 8(2):273--286, 1983.

\bibitem{ANZANELLO2011573}
M.~J. Anzanello and F.~S. Fogliatto.
\newblock Learning curve models and applications: Literature review and
  research directions.
\newblock {\em International Journal of Industrial Ergonomics}, 41(5):573--583,
  2011.

\bibitem{ApaG16}
R.~M. Apap and I.~E. Grossmann.
\newblock {Models and computational strategies for multistage stochastic
  programming under endogenous and exogenous uncertainties}.
\newblock {\em Computers {\&} Chemical Engineering}, 103:233--274, 2016.

\bibitem{Bir82}
J.~R. Birge.
\newblock The value of the stochastic solution in stochastic linear programs
  with fixed recourse.
\newblock {\em Mathematical programming}, 24(1):314--325, 1982.

\bibitem{ColM08}
M.~Colvin and C.~T. Maravelias.
\newblock {A stochastic programming approach for clinical trial planning in new
  drug development}.
\newblock {\em Computers {\&} Chemical Engineering}, 32(11):2626--2642, 2008.

\bibitem{ColM10}
M.~Colvin and C.~T. Maravelias.
\newblock {Modeling methods and a branch and cut algorithm for pharmaceutical
  clinical trial planning using stochastic programming}.
\newblock {\em European Journal of Operational Research}, 203(1):205--215,
  2010.

\bibitem{syngenta_cover}
P.~Comhaire and F.~Papier.
\newblock Syngenta uses a cover optimizer to determine production volumes for
  its european seed supply chain.
\newblock {\em Interfaces}, 45(6):501--513, 2015.

\bibitem{Fla10}
B.~{da Costa Flach}.
\newblock {\em {Stochastic Programming with Endogenous Uncertainty: An
  Application in Humanitarian Logistics}}.
\newblock PhD thesis, PUC-Rio, 2010.

\bibitem{DenFS10}
M.~Denizel, M.~Ferguson, and G.~e.~C. Souza.
\newblock Multiperiod remanufacturing planning with uncertain quality of
  inputs.
\newblock {\em IEEE Transactions on Engineering Management}, 57(3):394--404,
  2010.

\bibitem{DiaMP14}
M.~D{\'\i}az-Madro{\~n}ero, J.~Mula, and D.~Peidro.
\newblock A review of discrete-time optimization models for tactical production
  planning.
\newblock {\em International Journal of Production Research},
  52(17):5171--5205, 2014.

\bibitem{DonGXY22}
L.~Dong, X.~Geng, G.~Xiao, and N.~Yang.
\newblock Procurement strategies with unreliable suppliers under correlated
  random yields.
\newblock {\em Manufacturing \& Service Operations Management}, 24(1):179--195,
  2022.

\bibitem{EscGMP07}
L.~F. Escudero, A.~Gar{\'\i}n, M.~Merino, and G.~P{\'e}rez.
\newblock The value of the stochastic solution in multistage problems.
\newblock {\em Top}, 15(1):48--64, 2007.

\bibitem{EscGMU18}
L.~F. Escudero, M.~A. Gar{\'{i}}n, J.~F. Monge, and A.~Unzueta.
\newblock {On preparedness resource allocation planning for natural disaster
  relief under endogenous uncertainty with time-consistent risk-averse
  management}.
\newblock {\em Computers {\&} Operations Research}, 98:84--102, 2018.

\bibitem{GelV81}
L.~F. Gelders and L.~N. Van~Wassenhove.
\newblock Production planning: a review.
\newblock {\em European Journal of Operational Research}, 7(2):101--110, 1981.

\bibitem{GoeG04}
V.~Goel and I.~E. Grossmann.
\newblock {A stochastic programming approach to planning of offshore gas field
  developments under uncertainty in reserves}.
\newblock {\em Computers {\&} Chemical Engineering}, 28:1409--1429, 2004.

\bibitem{GoeG06}
V.~Goel and I.~E. Grossmann.
\newblock {A Class of stochastic programs with decision dependent uncertainty}.
\newblock {\em Mathematical Programming}, 108:355--394, 2006.

\bibitem{GroG04}
A.~Grosfeld-Nir and Y.~Gerchak.
\newblock Multiple lotsizing in production to order with random yields: Review
  of recent advances.
\newblock {\em Annals of Operations Research}, 126:43--69, 2004.

\bibitem{GuaP11}
Z.~Guan and A.~B. Philpott.
\newblock A multistage stochastic programming model for the new zealand dairy
  industry.
\newblock {\em International Journal of Production Economics}, 134(2):289--299,
  2011.

\bibitem{GupG11}
V.~Gupta and I.~E. Grossmann.
\newblock {Solution strategies for multistage stochastic programming with
  endogenous uncertainties}.
\newblock {\em Computers {\&} Chemical Engineering}, 35:2235--2247, 2011.

\bibitem{HaxC84}
A.~C. Hax and D.~Candea.
\newblock {\em {Production and Inventory Management}}.
\newblock Prentice-Hall, 1984.

\bibitem{HelW05}
H.~Held and D.~L. Woodruff.
\newblock {Heuristics for Multi-Stage Interdiction of Stochastic Networks}.
\newblock {\em Journal of Heuristics}, 11(5-6):483--500, 2005.

\bibitem{Hel16}
L.~Hellemo.
\newblock {\em {Managing Uncertainty in Design and Operation of Natural Gas
  Infrastructure}}.
\newblock Phd, Norwegian University of Science and Technology, 2016.

\bibitem{HelBT18}
L.~Hellemo, P.~I. Barton, and A.~Tomasgard.
\newblock Decision-dependent probabilities in stochastic programs with
  recourse.
\newblock {\em Computational Management Science}, 15(3):369--395, 2018.

\bibitem{HooM16}
F.~Hooshmand and S.~A. MirHassani.
\newblock {Efficient constraint reduction in multistage stochastic programming
  problems with endogenous uncertainty}.
\newblock {\em Optimization Methods and Software}, 31:359--376, 2016.

\bibitem{MooM18}
F.~Hooshmand and S.~A. MirHassani.
\newblock {Reduction of nonanticipativity constraints in multistage stochastic
  programming problems with endogenous and exogenous uncertainty}.
\newblock {\em Mathematical Methods of Operations Research}, 87(1):1--18, 2018.

\bibitem{JamYFXJ19}
A.~Jamalnia, J.-B. Yang, A.~Feili, D.-L. Xu, and G.~Jamali.
\newblock Aggregate production planning under uncertainty: a comprehensive
  literature survey and future research directions.
\newblock {\em The International Journal of Advanced Manufacturing Technology},
  102:159--181, 2019.

\bibitem{8643194}
T.~C.~J. Jeaunita, V.~Sarasvathi, M.~S. Harsha, B.~M. Bhavani, and
  T.~Kavyashree.
\newblock An automated greenhouse system using agricultural internet of things
  for better crop yield.
\newblock In {\em Smart Cities Symposium 2018}, pages 1--6, 2018.

\bibitem{JIN2018141}
X.~Jin, L.~Kumar, Z.~Li, H.~Feng, X.~Xu, G.~Yang, and J.~Wang.
\newblock A review of data assimilation of remote sensing and crop models.
\newblock {\em European Journal of Agronomy}, 92:141--152, 2018.

\bibitem{JohM74}
L.~A. Johnson and D.~C. Montgomery.
\newblock {\em {Operations research in production planning, scheduling and
  inventory control}}.
\newblock John Wiley and Sons Ltd, 1974.

\bibitem{jones2003managing}
P.~C. Jones, G.~Kegler, T.~J. Lowe, and R.~D. Traub.
\newblock Managing the seed-corn supply chain at syngenta.
\newblock {\em Interfaces}, 33(1):80--90, 2003.

\bibitem{jones2001matching}
P.~C. Jones, T.~J. Lowe, R.~D. Traub, and G.~Kegler.
\newblock Matching supply and demand: The value of a second chance in producing
  hybrid seed corn.
\newblock {\em Manufacturing \& Service Operations Management}, 3(2):122--137,
  2001.

\bibitem{JonWW98}
T.~W. Jonsbr{\aa}ten, R.~J.-B. Wets, and D.~L. Woodruff.
\newblock {A class of stochastic programs withdecision dependent random
  elements}.
\newblock {\em Annals of Operations Research}, 82:83--106, 1998.

\bibitem{Kaz04}
B.~Kazaz.
\newblock Production planning under yield and demand uncertainty with
  yield-dependent cost and price.
\newblock {\em Manufacturing \& Service Operations Management}, 6(3):209--224,
  2004.

\bibitem{KazW11}
B.~Kazaz and S.~Webster.
\newblock The impact of yield-dependent trading costs on pricing and production
  planning under supply uncertainty.
\newblock {\em Manufacturing \& Service Operations Management}, 13(3):404--417,
  2011.

\bibitem{KemKU11}
K.~G. Kempf, P.~Keskinocak, and R.~Uzsoy.
\newblock Planning production and inventories in the extended enterprise: A
  state of the art handbook, volume 1.
\newblock 2011.

\bibitem{Moutaz1994}
M.~Khouja and A.~Mehrez.
\newblock Economic production lot size model with variable production rate and
  imperfect quality.
\newblock {\em The Journal of the Operational Research Society},
  45(12):1405--1417, 1994.

\bibitem{KouXY21}
P.~Kouvelis, G.~Xiao, and N.~Yang.
\newblock Role of risk aversion in price postponement under supply random
  yield.
\newblock {\em Management Science}, 67(8):4826--4844, 2021.

\bibitem{LapL93}
G.~Laporte and F.~Louveaux.
\newblock {The integer L-shaped method for stochastic integer programs with
  complete recourse}.
\newblock {\em Operations Research Letters}, 13:133--142, 1993.

\bibitem{LauPK14}
M.~Laumanns, S.~Prestwich, and B.~Kawas.
\newblock {Distribution shaping and scenario bundling for stochastic programs
  with endogenous uncertainty}.
\newblock In J.~L. Higle, W.~R{\"{o}}misch, and S.~Sen, editors, {\em
  Stochastic Programming E-Print Series}, Stochastic Programming E-Print
  Series. Institut f{\"{u}}r Mathematik, 2014.

\bibitem{LeuN07}
S.~C. Leung and W.-L. Ng.
\newblock A stochastic programming model for production planning of perishable
  products with postponement.
\newblock {\em Production Planning and Control}, 18(3):190--202, 2007.

\bibitem{XiaYX15}
X.~Li, Y.~Li, and X.~Cai.
\newblock Remanufacturing and pricing decisions with random yield and random
  demand.
\newblock {\em Computers \& Operations Research}, 54:195--203, 2015.

\bibitem{KanZ18}
K.~Liu and Z.-H. Zhang.
\newblock Capacitated disassembly scheduling under stochastic yield and demand.
\newblock {\em European Journal of Operational Research}, 269(1):244--257,
  2018.

\bibitem{MagW12}
F.~Maggioni and S.~W. Wallace.
\newblock Analyzing the quality of the expected value solution in stochastic
  programming.
\newblock {\em Annals of Operations Research}, 200(1):37--54, 2012.

\bibitem{Mcc76}
G.~P. McCormick.
\newblock Computability of global solutions to factorable nonconvex programs:
  Part i—convex underestimating problems.
\newblock {\em Mathematical programming}, 10(1):147--175, 1976.

\bibitem{MerV11}
L.~Mercier and P.~{Van Hentenryck}.
\newblock {An anytime multistep anticipatory algorithm for online stochastic
  combinatorial optimization}.
\newblock {\em Annals of Operations Research}, 184(1):233--271, 2011.

\bibitem{MisU11}
H.~Missbauer and R.~Uzsoy.
\newblock Optimization models of production planning problems.
\newblock {\em Planning Production and Inventories in the Extended Enterprise:
  A State of the Art Handbook, Volume 1}, pages 437--507, 2011.

\bibitem{ModH55}
F.~Modigliani and F.~E. Hohn.
\newblock Production planning over time and the nature of the expectation and
  planning horizon.
\newblock {\em Econometrica, Journal of the Econometric Society}, pages 46--66,
  1955.

\bibitem{MulPGL06}
J.~Mula, R.~Poler, J.~P. Garc{\'\i}a-Sabater, and F.~C. Lario.
\newblock Models for production planning under uncertainty: A review.
\newblock {\em International journal of production economics}, 103(1):271--285,
  2006.

\bibitem{Pan21}
G.~Pantuso.
\newblock A node formulation for multistage stochastic programs with endogenous
  uncertainty.
\newblock {\em Computational Management Science}, 18(3):325--354, 2021.

\bibitem{PanB20}
G.~Pantuso and T.~K. Boomsma.
\newblock On the number of stages in multistage stochastic programs.
\newblock {\em Annals of Operations Research}, 292(2):581--603, 2020.

\bibitem{PeeSGV10}
S.~Peeta, F.~S. Salman, D.~Gunnec, and K.~Viswanath.
\newblock {Pre-disaster investment decisions for strengthening a highway
  network}.
\newblock {\em Computers {\&} Operations Research}, 37:1708--1719, 2010.

\bibitem{Pfl90}
G.~C. Pflug.
\newblock On-line optimization of simulated markovian processes.
\newblock {\em Mathematics of Operations Research}, 15(3):381--395, 1990.

\bibitem{SANA2010158}
S.~S. Sana.
\newblock An economic production lot size model in an imperfect production
  system.
\newblock {\em European Journal of Operational Research}, 201(1):158--170,
  2010.

\bibitem{TarG08}
B.~Tarhan and I.~E. Grossmann.
\newblock {A multistage stochastic programming approach with strategies for
  uncertainty reduction in the synthesis of process networks with uncertain
  yields}.
\newblock {\em Computers {\&} Chemical Engineering}, 32:766--788, 2008.

\bibitem{TarGG09}
B.~Tarhan, I.~E. Grossmann, and V.~Goel.
\newblock {Stochastic Programming Approach for the Planning of Offshore Oil or
  Gas Field Infrastructure under Decision-Dependent Uncertainty}.
\newblock {\em Industrial {\&} Engineering Chemistry Research}, 48:3078--3097,
  2009.

\bibitem{TarGG13}
B.~Tarhan, I.~E. Grossmann, and V.~Goel.
\newblock {Computational strategies for non-convex multistage MINLP models with
  decision-dependent uncertainty and gradual uncertainty resolution}.
\newblock {\em Annals of Operations Research}, 203(1):141--166, 2013.

\bibitem{9718402}
N.~N. Thilakarathne, H.~Yassin, M.~S.~A. Bakar, and P.~E. Abas.
\newblock Internet of things in smart agriculture: Challenges, opportunities
  and future directions.
\newblock In {\em 2021 IEEE Asia-Pacific Conference on Computer Science and
  Data Engineering (CSDE)}, pages 1--9, 2021.

\bibitem{TonFR12}
K.~Tong, Y.~Feng, and G.~Rong.
\newblock {Planning under Demand and Yield Uncertainties in an Oil Supply
  Chain}.
\newblock {\em Industrial {\&} Engineering Chemistry Research}, 51(2):814--834,
  2012.

\bibitem{VisSF04}
K.~Viswanath, P.~Srinivas, and S.~{F. Sibel}.
\newblock {Investing in the Links of a Stochastic Network to Minimize Expected
  Shortest Path Length}.
\newblock 2004.

\bibitem{VosW06}
S.~Voss and D.~L. Woodruff.
\newblock {\em Introduction to computational optimization models for production
  planning in a supply chain}, volume 240.
\newblock Springer Science \& Business Media, 2006.

\bibitem{9167626}
P.~B. Wakhare, S.~Neduncheliyan, and G.~S. Sonawane.
\newblock Automatic irrigation system based on internet of things for crop
  yield prediction.
\newblock In {\em 2020 International Conference on Emerging Smart Computing and
  Informatics (ESCI)}, pages 129--132, 2020.

\bibitem{XieMG21}
L.~Xie, J.~Ma, and M.~Goh.
\newblock Supply chain coordination in the presence of uncertain yield and
  demand.
\newblock {\em International Journal of Production Research},
  59(14):4342--4358, 2021.

\bibitem{YanL95}
C.~A. Yano and H.~L. Lee.
\newblock Lot sizing with random yields: A review.
\newblock {\em Operations research}, 43(2):311--334, 1995.

\bibitem{XiaGYL21}
X.~Yuan, G.~Bi, Y.~Fei, and L.~Liu.
\newblock Supply chain with random yield and financing.
\newblock {\em Omega}, 102:102334, 2021.

\end{thebibliography}

\begin{appendices}
\section{Linearized PP-DESUP} \label{app:linearized_pp_desup}

  In this appendix we provide an exact linearization of the PP-DESUP using McCormick inequalities. Particularly, we add the variable $\mu_{pds}$ to replace the product $\delta_{pd}w_{pds}$ and the variable $\rho_{pds}$ to replace the product $\delta_{pd}o_{pds}$. To define the constraints that enforce $\mu_{pds}=\delta_{pd}w_{pds}$ we note that $D_{pds}$ is an upper bound on the decision variable $\mu_{pds}$ as more of a product cannot be sold in a scenario than there is demand.
  To define the constraints that enforce $\rho_{pds}=\delta_{pd}o_{pds}$ we note that $o_{pds}\leq N_{pds}=\sum_{f\in\mathcal{F}}Y_{pfds}\min\{B_{f},\max_{l\in\mathcal{L}_{pf}}U_{pfl}\}$ on $o_{pds}.$ This upper bound reflects that the inventory in excess of demand for a product in a scenario can never exceed the inventory achieved when allocating the maximum amount to each facility and the corresponding yields occurring. 
The resulting linearized model is as follows.

\begin{subequations}
\small
\label{eq:fullmodel}
\begin{align}
\label{eq:fullmodel:obj}v_{PP-DESUP}^{*} = \max~&-\sum_{p\in\mathcal{P}}\sum_{f\in\mathcal{F}}C_{pf}x_{pf}+\sum_{p\in\mathcal{P}}\sum_{d\in\mathcal{D}_p}\sum_{s\in\mathcal{S}_{pd}}\pi_{sd}\bigg(P_p \mu_{pds}+O_{p}\rho_{pds}\bigg)\\
\label{eq:model:c1b} & \sum_{p\in\mathcal{P}}x_{pf} \leq B_f & \forall f\in\mathcal{F}\\
\label{eq:fullmodel:c2} &   \sum_{l\in\mathcal{L}_{pf}}y_{pfl} =1&\forall p\in\mathcal{P}, f\in\mathcal{F}  \\
\label{eq:fullmodel:c1} &   \sum_{l\in\mathcal{L}_{pf}}L_{pfl}y_{pfl} \leq x_{pf}\leq \sum_{l\in\mathcal{L}_{pf}}\min\{b_{f},U_{pfl}y_{pfl}\}&\forall p\in\mathcal{P}, f\in\mathcal{F}\\ 
\label{eq:fullmodel:c4}    & \sum_{d\in\mathcal{D}_p}\delta_{pd} = 1 & \forall p\in\mathcal{P} \\
\label{eq:fullmodel:c3} &\sum_{f\in\mathcal{F}}y_{p,f,l(p,f,d)}\geq |\mathcal{F}|\delta_{pd} &\forall p\in\mathcal{P}, d\in \mathcal{D}_p\\
\label{eq:fullmodel:c0} &z_{pds}= \sum_{f\in\mathcal{F}}Y_{pfds}x_{pf} & \forall p\in\mathcal{P}, d\in\mathcal{D}_p, s\in\mathcal{S}_{pd}, \\
\label{eq:fullmodel:c6}  &w_{pds}\leq D_{pds}&\forall p\in\mathcal{P}, d\in\mathcal{D}_p, s\in\mathcal{S}_{pd}\\
\label{eq:fullmodel:c5}  &w_{pds} + o_{pds}= z_{pds} & \forall p\in\mathcal{P}, d\in\mathcal{D}_p, s\in\mathcal{S}_{pd}\\
\label{eq:fullmodel:fs_dvx} & x_{pf} \geq 0  &\forall p\in\mathcal{P}, f\in\mathcal{F}, \\ 
\label{eq:fullmodel:fs_dvy} & y_{pfl} \in \{0,1\} &\forall p\in\mathcal{P}, f\in\mathcal{F}\, l \in L_{pf}, \\
\label{eq:fullmodel:fs_dvd} & \delta_{pd} \in \{0,1\} &\forall p\in\mathcal{P}, d \in \mathcal{D}_{p}, \\
\label{eq:fullmodel:ss_dvs} & z_{pds} \geq 0, w_{pds} \geq 0, o_{pds} \geq 0 & \forall p \in \mathcal{P}, d \in \mathcal{D}_{p}, s \in \mathcal{S}_{pd}. \\
\label{eq:fullmodel:c7}            &\mu_{pds}\leq w_{pds}&\forall p\in\mathcal{P}, d\in\mathcal{D}_p, s\in\mathcal{S}_{pd}\\   
\label{eq:fullmodel:c8}            &\mu_{pds}\leq D_{pds}\delta_{pd}&\forall p\in\mathcal{P}, d\in\mathcal{D}_p, s\in\mathcal{S}_{pd}\\
\label{eq:fullmodel:c9}            &\mu_{pds}\geq w_{pds} - D_{pds}(1-\delta_{pd})&\forall p\in\mathcal{P}, d\in\mathcal{D}_p, s\in\mathcal{S}_{pd}\\
\label{eq:fullmodel:c10}            &\rho_{pds}\leq o_{pds}&\forall p\in\mathcal{P}, d\in\mathcal{D}_p, s\in\mathcal{S}_{pd}\\   
\label{eq:fullmodel:c11}            &\rho_{pds}\leq N_{pds}\delta_{pd}&\forall p\in\mathcal{P}, d\in\mathcal{D}_p, s\in\mathcal{S}_{pd}\\
\label{eq:fullmodel:c12}            &\rho_{pds}\geq o_{pds} - N_{pds}(1-\delta_{pd})&\forall p\in\mathcal{P}, d\in\mathcal{D}_p, s\in\mathcal{S}_{pd}.
\end{align}
\end{subequations}

\end{appendices}

\end{document}